\documentclass{amsart}
\usepackage[english]{babel}
\usepackage{amsfonts, amsmath, amsthm, amssymb,amscd,indentfirst}

\usepackage{xcolor}
\usepackage{MnSymbol}
\usepackage{esint}
\usepackage{amsmath,amssymb,latexsym,indentfirst}
\usepackage{times}
\usepackage{palatino}

\newtheorem{theorem}{Theorem}[section]

\newtheorem{proposition}{Proposition}[section]

\newtheorem{definition}{Definition}[section]

\newtheorem{corollary}{Corollary}[section]

\newtheorem{example}{Example}[section]
\newtheorem{remark}{Remark}[section]
\newtheorem{assumption}{Assumption}[section]

\begin{document}

\title[Heat conservation]{Heat conservation for generalized Dirac Laplacians on manifolds with boundary}


\author{Levi Lopes de Lima}

\thanks{Universidade Federal do Cear\'a,
	Departamento de   Matem\'atica, Campus do Pici, R. Humberto Monte, s/n, 60455-760,
	Fortaleza/CE, Brazil.}

\email{levi@mat.ufc.br}

\thanks{Partially suported by  CNPq/Brazil grant 311258/2014-0 and FUNCAP/CNPq/PRONEX Grant 00068.01.00/15.}

\maketitle

\begin{abstract}
We consider  a notion of conservation for the heat semigroup associated to a generalized Dirac  Laplacian acting on sections of a vector bundle over a noncompact manifold with a (possibly noncompact) boundary under mixed boundary conditions. Assuming that the geometry of the underlying manifold is controlled in a suitable way and imposing uniform lower bounds  on the zero order (Weitzenb\"ock) piece  of the Dirac Laplacian  and on the endomorphism defining the mixed boundary condition we show that the corresponding conservation principle holds. 
A key ingredient in the proof is a domination property for the heat semigroup which follows from an extension to this setting of a Feynman-Kac formula recently proved in \cite{dL1} in the context of differential forms.
When applied to the Hodge Laplacian acting on differential forms satisfying absolute boundary conditions, this extends previous results by Vesentini \cite{Ve} and Masamune \cite{M} in the boundaryless case. Along the way we also prove a vanishing result for $L^2$ harmonic sections in the broader context of generalized (not necessarily Dirac) Laplacians.
These  results are  further 
illustrated with applications to the Dirac Laplacian acting on spinors and to the Jacobi operator acting on sections of the normal bundle of a free boundary minimal immersion. 
\end{abstract}

\section{Introduction and statement of the main result}\label{intro}

Throughout this note we consider a noncompact, oriented Riemannian manifold $(X,g)$ of dimension $n\geq 2$. We assume that $X$ carries a (possibly noncompact) boundary $\Sigma$, on which an inwardly oriented unit normal vector $\nu$ is globally defined. Also, we assume that $X$ is geodesically complete in the sense that any geodesic  avoiding $\Sigma$ is defined for all time. We denote by $d_X$ the intrinsic distance on $X$, by $\nabla$ the Levi-Civita connection on tensors on $X$ and by $B=-\nabla_\nu$ the shape operator of $\Sigma$.

Let $\mathcal E\to X$ be a Riemannian (or Hermitean) vector bundle endowed with a fiber metric $\langle\,,\rangle$ and a compatible connection, still denoted by  $\nabla$.
Recall that a {\em generalized Laplacian} acting on sections of  $\mathcal E$ is a second order elliptic operator given by
\[
\Delta=\nabla^*\nabla+W,
\]
where $\nabla^*\nabla$ is the Bochner Laplacian associated to $\nabla$ and $W\in \Gamma(X,{\rm End}(\mathcal E))$ is pointwise selfadjoint bundle endomorphism.  
We will refer to $W$ as the {\em Weitzenb\"ock operator}. Also, we denote the standard functional spaces of sections of $\mathcal E$ by $L^p(X,\mathcal E)$, etc.

In the presence of $\Sigma$ we need to attach to $\Delta$ suitable boundary conditions of elliptic type. Here we adopt a certain class of {mixed} boundary conditions which are determined by an orthogonal decomposition
\[
\mathcal E|_{\Sigma}=\mathcal F_+\oplus\mathcal F_-
\]
corresponding to the eigenbundles of a selfadjoint involution $\mathcal I\in\Gamma(X,{\rm End}(\mathcal E|_\Sigma))$ and a pointwise selfadjoint endomorphism $S\in\Gamma(\Sigma,{\rm End}(\mathcal F_+))$; see Section \ref{feyn-kac}. Also, we assume throughout the text that both  
\[
W\in{L^2_{\rm loc}}(X,{\rm End}(\mathcal E)) \quad {\rm and}\quad  S\in{L^2_{\rm loc}}(\Sigma,{\rm End}(\mathcal F_+))
\]
are uniformly bounded from below. These requirements are better stated in terms of the functions 
\begin{equation}\label{defr}
w:X\to \mathbb R, \quad w(x)=\inf_{|\phi|=1}\langle W(x)\phi,\phi\rangle, 
\end{equation}
and 
\begin{equation}\label{defsigma}
\sigma:\Sigma\to\mathbb R, \quad \sigma(x)=
\inf_{|\phi|=1}\langle S(x)\phi,\phi\rangle.
\end{equation}

\begin{assumption}\label{assumpwb}
	There exist constants $c_1,c_2>-\infty$ such that $w\geq c_1$
and $\sigma\geq c_2$. 
	\end{assumption}

Under this assumption and imposing mixed  boundary conditions as above, $\Delta$ admits a natural selfadjoint extension which we denote by $\Delta_{W,S}$. Hence, we may apply the spectral theorem to define the corresponding heat semigroup 
\[
e^{-\frac{1}{2}t\Delta_{W,S}}: {L^2}(M,\mathcal E)\to {L^2}(M,\mathcal E), \qquad t>0.
\]
In this setting, we denote by  $\mathcal D_S(\mathcal E)$ the space of compactly supported, smooth sections and by $\mathcal H(\mathcal E)$ the space of harmonic sections (i.e. sections lying in $\ker \Delta_{W,S}$), where in both cases we assume that the  given mixed boundary conditions are met.
Also, $(\,,)$ will denote the standard $L^2$ pairing between sections of $\mathcal E$.

The definition below is motivated by \cite{Ve,M}, where it is discussed in the context of differential forms on boundaryless manifolds.

\begin{definition}\label{consforms}
	Under the conditions above, we say that the heat conservation principle holds for $\Delta_{W,S}$ if the equality
	\begin{equation}\label{consforms2}
	\left(e^{-\frac{1}{2}t\Delta_{W,S}}\phi,\eta\right)=\left(\phi,\eta\right), \qquad t>0,
	\end{equation}	
	holds
	for any  $\phi\in\mathcal D_S(\mathcal E)$ and any $\eta\in \mathcal H(\mathcal E)\cap {L^\infty}(X,\mathcal E)$. 
\end{definition}

This means that bounded harmonic sections are preserved by the heat semigroup. When $\mathcal E$ is the trivial line bundle, $\Delta=\Delta_0$, the (nonnegative) Laplacian acting on functions, and we impose Neumann boundary conditions, this boils down to requiring that $X$ is stochastically complete (with respect to normally reflected Brownian motion); see Section \ref{conserref} for a discussion of this point. Thus, Definition \ref{consforms} is a straightforward generalization of a much studied property of a natural  diffusion process on manifolds with boundary.

Our main result provides a simple criterium for the validity of this principle. 
For technical reasons we need to control the geometry of the underlying manifold  $(X,g)$ both at infinity and around the boundary. Thus, througout the text we assume that the following holds.

\begin{assumption}\label{assump} The Ricci tensor ${\rm Ric}$ is bounded from below and 
	\begin{itemize}
		\item  Either $\Sigma$ is convex (i.e. $B\geq 0$);
		\item Or
		\begin{enumerate}
			\item $B$ is bounded;
			\item there exists $r_0>0$ such that the geodesic collar map 
			\[
			\Lambda_{r_0}: [0,r_0)\times\Sigma\to X,\quad \Lambda_{r_0}(r,x)=\exp_x(r\nu),
			\]
			is a diffeomorphism onto its image;
			\item the sectional curvature is uniformly bounded from above on the image of $\Lambda_{r_0}$.   
		\end{enumerate}
	\end{itemize}
	\end{assumption}

This kind of assumption appears in \cite[Section 3.2.3]{W}. As proved in \cite[Theorem 3.2.9]{W}, it leads to an integrability result for the exponentiated boundary local time associated to reflected Brownian motion; see Theorem \ref{intloct}.  Another useful consequence of Assumption \ref{assump} is that $X$ is stochastically complete in the sense that the sample paths of the reflected Brownian motion remain in $X$ for any positive time; see Theorem \ref{consesp}. 

We need a further  specialization on the structure of $\Delta$. Recall that a {\em Dirac operator} on $\mathcal E$ is a first order differential operator such that $D^2$ is a generalized Laplacian. 
We then say that $\Delta=D^2$ is a {\em generalized Dirac Laplacian}. 
We note that the existence of $D$ is equivalent to requiring that $\mathcal E$ is a Dirac bundle with respect to which $D$ is the corresponding Dirac operator \cite[Proposition 10.1.5]{Ni}. 
In particular, we have the Leibniz rule
\begin{equation}\label{fundrel}
D(\xi\cdot\phi)=D_{\sf c}\xi\cdot\phi+\xi\cdot D\phi,
\end{equation}
for $\phi\in\Gamma(X,\mathcal E)$, $\xi\in\Gamma(X,{\sf Cl}(TX))$, where ${\sf Cl}(TX)$ is the Clifford bundle of $(X,g)$, the dot is Clifford multiplication and $D_{\sf c}$ is the Dirac operator on ${\sf Cl}(TX)$, viewed as a Dirac bundle over itself under left Clifford multiplication \cite[Chapter II, Example 5.8]{LM}.

\begin{assumption}\label{assdirac}
	There holds $\mathcal H(\mathcal E)\cap L^\infty(X,\mathcal E)\subset {\rm ker}\, D$. In other words, any bounded harmonic section $\phi\in \Gamma(X,\mathcal E)$ meeting the given mixed boundary conditions satisfies $D\phi=0$.
	\end{assumption}

With this terminology at hand we can state our main result. 

\begin{theorem}\label{main}
	If $(X,g)$ satisfies Assumption \ref{assump} and a generalized Dirac Laplacian $\Delta=D^2$ acting on sections of $\mathcal E\to X$ satisfies Assumptions \ref{assumpwb} and \ref{assdirac} then the heat conservation principle holds for $\Delta_{W,S}$.
\end{theorem}

This paper is organized as follows. In Section \ref{conserref} we review the properties of Brownian motion and  Brownian bridge in the reflected case and in Section \ref{feyn-kac} we discuss mixed boundary conditions. The proof of Theorem \ref{main} is included  in Section \ref{proofmain} and makes use of a Feynman-Kac formula (Theorem \ref{mainfk}), which allows us to obtain a path integral representation for the heat kernel associated to $e^{-\frac{1}{2}t\Delta_{W,S}}$ (Theorem \ref{desintfin}). This is a key step in establishing the corresponding semigroup domination property (Theorem \ref{domsem0} and Corollary \ref{domsem}). We stress that since the proof of this property does not require the use of Assumption \ref{assdirac}, it holds for {\em any} generalized Laplacian $\Delta_{W,S}$ satisfying Assumption \ref{assumpwb}. In particular, we are able to obtain a vanishing result in this rather general setting (Corollary \ref{vanishres}).  
Finally, in Section  \ref{examples} we discuss  applications of our  results to certain generalized Laplacians appearing in Geometry, namely, the Hodge Laplacian acting on differential forms, the Dirac Laplacian acting on spinors and the Jacobi operator acting on sections of the normal bundle of a free boundary minimal submanifold.

Finally, we mention that a preliminary version of this article, with a sketch of the proof of our main result in the context of the  Hodge Laplacian, has been published in \cite{dL3}. 

\section{Preliminary results on  reflected Brownian motion}\label{conserref}

In this section we collect a few technical results on the reflected Brownian motion on the underlying Riemannian manifold $(X,g)$. Besides reviewing the stochastic notions need\-ed in the sequel, this is intended to justify the claim in the Introduction that Definition \ref{consforms} can be viewed as a natural generalization of $X$ being stochastically complete with respect to this diffusion process. We then discuss the associated reflected Brownian bridge, which happens to be a key ingredient in establishing a path integral representation for the heat semigroup $e^{-\frac{1}{2}t\Delta_{W,S}}$. 

Let ${\sf X}_t^x$ be reflected Brownian motion starting at $x\in X$ \cite{AL, IW, Hs2, dL1, W}. This is a continuous stochastic process driven by  $-\frac{1}{2}\Delta_0$, where $\Delta_0$ is the (nonnegative) Laplacian  acting on bounded functions satisfying  Neumann boundary condition along $\Sigma$\footnote{Thus, our sign convention is so that  $\Delta_0=-d^2/dx^2$ on $\mathbb R$.}.   
Recall that ${\sf X}_t=\pi {\widetilde {\sf X}}_t$, where $\pi:P_{\rm SO}(X)\to X$ is the principal bundle of oriented orthonormal frames and ${\widetilde {\sf X}}_t$ is the {\em horizontal} reflected Brownian motion starting at some $\widetilde x\in\pi^{-1}(x)$, whose anti-development is the standard Brownian motion $b_t$ in $\mathbb R^n$.
Formally, ${\widetilde {\sf X}}_t$ satisfies the stochastic differential equation 
\begin{equation}
\label{stocut}
d{\widetilde {\sf X}}_t=\sum_{i=1}^nH_i({\widetilde {\sf X}}_t)\circ db^i_t+\nu^\dagger({\widetilde {\sf X}}_t)d\lambda_t,
\end{equation}
where $\{H_i\}_{i=1}^n$ are the fundamental horizontal vector fields on  $P_{\rm SO}(X)$, the dagger means the standard equivariant lift (scalarization) of tensor fields on $X$ to $P_{{\rm SO}}(X)$ and 
$\lambda_t$ is the boundary local time associated to ${\sf X}_t$. 
We recall that $\lambda_t$ is a nondecreasing process which only increases when the Brownian path hits the boundary.
	
In general, ${\sf X}_t$ might fail to be a Markov process. More precisely, let $\widehat X=X\cup\{\infty\}$ be the one-point compactification of the pair $(X,\Sigma)$ and define 
\[
{\bf e}(x)=\inf\{t\geq 0;{\sf X}^x_t=\infty\}, \quad x\in X.
\]
For obvious reasons, $\bf e$ is called the extinction time of ${\sf X}_t$. 
Now, the Markov property for ${\sf X}_t$ might not hold precisely because the process might be explosive in the sense that $\bf e \not\equiv+\infty$. 

This somewhat annoying explosiveness property can be reformulated in analytical terms as follows. A  version of the Feynman-Kac formula in this setting says that  
the (local) semigroup generated by $-\frac{1}{2}\Delta_0$ is given by 
\begin{equation}\label{defsemi}
(e^{-\frac{1}{2}t\Delta_0}f)(x)=\mathbb E_x[f({\sf X}_t^x)\chi_{\{t<{\bf e}(x)\}}],
\end{equation}
where $\mathbb E_x$ is the expectation associated to the law $\mathbb P_x$ of ${\sf X}_t^x$, $f\in L^2(X)\cap L^\infty(X)$ satisfies Neumann boundary condition  and
$\chi$ is the indicator function. 
It follows that $t\mapsto e^{-\frac{1}{2}t\Delta_0}$ is a positive preserving, contraction semigroup on the space of all such functions, so by interpolation it can be extended as a contraction semigroup to $L^p(X)$, $1\leq p\leq \infty$. Thus, we may  
apply (\ref{defsemi}) with $f={\bf 1}$, the function identically equal to $1$, in order to get 
\begin{equation}\label{f-kfunct}
(e^{-\frac{1}{2}t\Delta_0}{\bf 1})(x)=\mathbb P[t<{\bf e}(x)]. 
\end{equation}
So in general we have $e^{-\frac{1}{2}t\Delta_0}{\bf 1}\leq {\bf 1}$ and being  explosive means precisely that  
$e^{-\frac{1}{2}t\Delta_0}{\bf 1}\not\equiv {\bf 1}$ for some (and hence any) $t>0$. This means that constant functions are {\em not} preserved by the semigroup.

Another way of expressing this sub-Markov property of ${\sf X}_t$ relies on the well-known fact
that the semigroup action can be represented by convolution against a smooth kernel. More precisely, 
\[
(e^{-\frac{1}{2}t\Delta_0}f)(x)=\int_X K_0(t;x,y)f(y)dX_y,
\]   
where $K_0$ is the Neumann heat kernel, that is, the fundamental solution of the initial value problem associated to the heat operator 
\[
L=\frac{\partial }{\partial t}+\frac{1}{2}\Delta_0
\]    
with Neumann boundary condition along $\Sigma$.
Thus, by (\ref{f-kfunct}) in general we have 
\[
\int_X K_0(t;x,y)dX_y\leq 1,
\]
and we see once again that in the explosive case the strict equality holds 
for some $t>0$. Thus, in general we are not allowed to interpret $K_0$ as a transition probability density function for ${\sf X}_t$.

The following well-known proposition summarizes the discussion above.
Here, $(\,,)_0$ is the standard $L^2$ pairing on functions.

\begin{proposition}\label{equiv}
	The following are equivalent:
	\begin{enumerate}
		\item ${\sf X}_t$ is non-explosive in the sense that $\bf e\equiv +\infty$;
			\item For some/any $t>0$ and any $x\in X$, $K_0(t;x,\cdot)$ is a probability density function on $X$. 
		\item For some/any $t>0$, $ e^{-\frac{1}{2}t\Delta_0}{\bf 1}={\bf 1}$;
		\item For some/any $t>0$, $(e^{-\frac{1}{2}t\Delta_0}f,{\bf 1})_0=( f,{\bf 1})_0$, for any compactly supported function  $f$ on $X$ satisfying Neumann boundary condition.
	\end{enumerate} 
	\end{proposition}

We now recall a standard terminology.  

\begin{definition}\label{conserv0}
	If any of the conditions in Proposition \ref{equiv} happens then we say  that $X$ is stochastically complete.
\end{definition}

The validity of this property means that the desired probabilistic interpretation for $K_0$ has been restored  so that  ${\sf X}_t$ is turned into a genuine Markov process. Equivalently, constant functions are preserved by the associated semigroup.
Also, in view of item (4) we see that $X$ being stochastically complete is equivalent to the heat conservation principle holding for $\Delta_0$. This provides the link between this classical notion and our Definition \ref{consforms}.

It is not hard to exhibit examples of noncompact, geodesically complete manifolds which fail to be stochastically complete; see \cite{Gr} for a rather complete survey in the boundaryless case. On the other hand, a  celebrated criterium due to Gregor'yan  
\cite[Theorem 9.1]{Gr}, which certainly can be adapted to our setting,  provides a sufficient condition for stochastic completeness in terms of volume growth.
However, from our viewpoint it is natural to consider instead the following test which involves imposing curvature bounds both in the interior and along the boundary. In the boundaryless case, where only the lower bound on the Ricci tensor is required, this is due to Yau \cite{Y}.

\begin{theorem}\label{consesp}
	If Assumption \ref{assump}  is satisfied then  $X$ is stochastically complete.
	 \end{theorem}
 
\begin{proof}
 	See Remark \ref{function} for a simple proof based on the semigroup domination property proved in Section \ref{semigroup}. 
\end{proof}

As already mentioned, Assumption \ref{assump} also yields  an integrability result for the boundary local time $\lambda_t$. Clearly, we may assume that the lower bound for $B$, say $\underline\kappa$, is negative. 

\begin{theorem}\cite[Theorem 3.2.9]{W}\label{intloct}
	If Assumption \ref{assump} holds then for any $p\in[1,+\infty)$ there exist $K_1^{(p)},K_2^{(p)}> 0$ such that   
	\[
	\mathbb E_x[e^{-p\underline \kappa\lambda_t}]\leq K_1^{(p)}{e^{K_2^{(p)}t}}, 
	\]
	for all $t\geq 0$ and $x\in X$. 
	\end{theorem}

We now turn to the so-called {\em reflected Brownian bridge} associated to ${\sf X}_t$; see \cite[Appendix A]{dL2} for details. For each $t>0$ and $x,y\in X$, this is the process $\mathfrak X_{s;x,y}$, $0\leq s\leq t$, which starts at $x$, follows the reflected Brownian motion ${\sf X}_s^x$ and is further conditioned to hit $y$ in time $t$. At least for $0\leq s<t$, it is immediate to check that its law $\mathbb P_{t;x,y}$ satisfies 
\begin{equation}\label{lawrel}
\frac{d\mathbb P_{t;x,y}}{d\mathbb P_x}|_{\mathcal G_s}=\frac{K_0(t-s;{\sf X}_s^x,y)}{K_0(t;x,y)},
\end{equation} 
where $\mathcal G_s$ is the standard filtration associated to ${\sf X}_t$.
It then follows that the reflected Brownian bridge is just reflected Brownian motion with an added drift involving the logarithmic derivative of $K_0$. In particular, $\mathfrak X_{s;x,y}$ is a $\mathbb P_{t;x,y}$-semimartingale in the range $0\leq s<t$. It is crucial in applications to be able to extend  this property to $s=t$.  

\begin{proposition}\label{extprop}
If Assumption \ref{assump} holds then reflected Brownian bridge $\mathfrak X_{s;x,y}$ is a $\mathbb P_{t;x,y}$-semimartingale in the whole interval $[0,t]$. 
	\end{proposition}

\begin{proof}
	We only sketch the proof, as it follows by adapting standard results in the available literature for the boundaryless case. First, as explained in \cite[Section 3.2.3]{W}, Assumption \ref{assump} implies that, by eventually passing to a conformally deformed metric, we may  assume that $\Sigma$ is convex. This guarantees that any two points in $X$ can be joined by at least one minimizing geodesic. By using standard comparison theory, this implies that, at least locally, we have at our disposal the usual package of geometric bounds, which includes the Bishop-Gromov inequality, the doubling volume property and Gaussian bounds for $K_0$, where the controlling constants entering in theses estimates depend only on the local geometry; see \cite[Appendix A]{Gu1}. We then argue as in \cite[Appendix B]{Gu1} to obtain a localized gradient estimate for $\log K_0$, which adapts an argument in \cite{AT}. With these informations at hand, we can easily establish local estimates of the types
	\[
	D_1t^{-n/2}e^{-D_2\frac{d_X(x,y)^2}{t}}\leq K_0(t;x,y)\leq D_3t^{-n/2}e^{-D_4\frac{d_X(x,y)^2}{t}},
	\]       
and 
\[
|\nabla \log K_0(t;x,y)|\leq D_5\left(t^{-1/2}+t^{-1}d_X(x,y)\right),
\]
where the constants $D_j$ only depend on the local geometry of $X$; cf. \cite[Proposition {2.8}]{Gu1}. From this point we may proceed as in the proof of \cite[Theorem 2.7]{Gu1} to check that the following localized inequality holds:
\[
\mathbb E_{t;x,y}\left[\int_0^t\left|\nabla \log K_0(t-s;\mathfrak X_{s;x,y},y)\right|\left|\nabla f(\mathfrak X_{s;x,y})\right|ds\right ]<+\infty, 
\]  
where
$\mathbb E_{t;x,y}$ is the expectation associated to  $\mathbb P_{t;x,y}$ and 
the compactly supported function $f$ is supposed to satisfy Neumann boundary conditions in case ${\rm supp}\,f\cap\Sigma\neq\emptyset$. As explained in \cite{Gu1}, this suffices to complete the proof. 
\end{proof}



\section{Mixed boundary conditions for generalized Laplacians}\label{feyn-kac}

Rather complete studies of elliptic boundary conditions for generalized Laplacians, including the delicate issue of the existence and explicit computation of the corresponding heat kernel asymptotics, can be found in  the available literature; see \cite{AE, Gi, Gru} for instance. Here we single out a class of such boundary conditions  which suffices for the applications we have in mind.

We start with a pointwise selfadjoint involution $\mathcal I\in\Gamma(X,{\rm End}(\mathcal E|_\Sigma))$, which we extend to a collared neighborhood of $\Sigma$ so that $\nabla_\nu\mathcal I=0$. Let 
\[
\Pi_{\pm}=\frac{1}{2}\left(I\pm\mathcal I\right)
\]
be the corresponding projections onto the eigenbundles $\mathcal F_\pm=\Pi_\pm\mathcal E|_\Sigma$ of $\mathcal I$. Clearly,
\[
\nabla_\nu\Pi_\pm=\Pi_\pm\nabla_\nu.
\]
Now take a pointwise selfadjoint endomorphism $S\in\Gamma(\Sigma,{\rm End}(\mathcal F_+))$ and extend it to $\mathcal E|_\Sigma$ by declaring that $S=0$ on $\mathcal F_-$. We may assume that the extension of $S$ to the collared neighborhood, still denoted $S$, satisfies $\nabla_\nu S=0$. It then follows that 
\[
S\Pi_\pm=\Pi_\pm S. 
\]

\begin{definition}\label{mixedcond}
	A section $\phi\in\Gamma(\mathcal E)$ satisfies { mixed boundary conditions} if its restriction to $\Sigma$, still denoted $\phi$, satisfies
	\begin{equation}\label{mixedcond2}
	\Pi_+(\nabla_\nu-S)\phi=0,\quad \Pi_-\phi=0.
	\end{equation}
	\end{definition}

The qualification ``mixed'' of course is due to the fact that this kind of boundary condition is Dirichlet in the $\mathcal F_-$-direction and Robin in the $\mathcal F_+$-direction. This seems to be the largest class of local elliptic boundary conditions to which the stochastic methods in Section \ref{semigroup} apply; see Remark \ref{bdconditions} below. The relevance of mixed boundary conditions in Quantum Field Theory is explained in \cite{AE,Va}. 

For the next proposition, recall that $\mathcal D_S(\mathcal E)$ is the space of smooth, compactly supported sections satisfying (\ref{mixedcond2}). 

\begin{proposition}\label{propmixed} 
If a generalized Laplacian $\Delta$ satisfies Assumption \ref{assumpwb} with $S$ as in (\ref{mixedcond2}) then the bilinear form
\[
Q:\mathcal D_S(\mathcal E)\times \mathcal D_S(\mathcal E)\to\mathbb R,\quad Q(\phi,\eta)=\int_X\langle\Delta \phi,\eta\rangle dX,
\]
is symmetric and bounded from below. 
\end{proposition}

\begin{proof}
	By adding a sufficiently large positive multiple of the identity to $S$ we may assume that $c_2\geq 0$.
	Recall that the Bochner Laplacian is locally given  by 
	\[
	\nabla^*\nabla=-\sum_{i=1}^n\left(\nabla_{e_i}\nabla_{e_i}-\nabla_{\nabla_{e_i}e_i}\right). 
	\]
	By choosing the orthonormal frame $\{e_i\}$ so that $\nabla_{e_i}e_j=0$ at the given point and defining a vector field $Z$ on $M$ by $\langle Z,Y\rangle=\langle\nabla_Y\phi,\eta\rangle$ we have 
	\[
	{\rm div}\, Z= \sum_ie_i\langle\nabla_{e_i}\phi,\eta\rangle=-\langle\nabla^*\nabla\phi,\eta\rangle+\langle\nabla\phi,\nabla\eta\rangle,
	\]
	so that
	\begin{equation}\label{intpart}
	\int_M\langle\nabla^*\nabla\phi,\eta\rangle\,dM=\int_M\langle\nabla\phi,\nabla\eta\rangle\,dM+\int_\Sigma\langle\nabla_\nu\phi,\eta\rangle\,d\Sigma,
	\end{equation}
	But
	\begin{eqnarray*}
		\int_\Sigma\langle\nabla_\nu\phi,\eta\rangle\,d\Sigma 
		& = & \int_\Sigma\langle\nabla_\nu(\Pi_+\phi+\Pi_-\phi),\Pi_+\eta+\Pi_-\eta\rangle d\Sigma \\
		& = & \int_\Sigma \langle\Pi_+\nabla_\nu\phi,\Pi_+\eta\rangle
	 d\Sigma\\
	& = & 	\int_\Sigma\langle\Pi_+(\nabla_\nu- S)\phi,\Pi_+\eta\rangle
	 d\Sigma +\int_\Sigma\langle\Pi_+ S\phi,\Pi_+\eta\rangle d\Sigma,
	\end{eqnarray*}
so that 
\begin{equation}\label{form}
 Q(\phi,\eta)=\int_X\langle\nabla\phi,\nabla\eta\rangle dX+\int_X\langle W\phi,\eta\rangle dX+\int_\Sigma\langle S\phi,\eta\rangle d\Sigma.
\end{equation}
From this it is immediate that $Q$ is both symmetric and bounded from below. 
\end{proof}

Let us take $W$ and $S$ as in Theorem \ref{main}.  Thus, Proposition \ref{propmixed} applies and  the  quadratic form $Q$, which is associated to the densely defined unbounded operator $\Delta:\mathcal D_S(\mathcal E)\subset {L^2}(X,\mathcal E)\to {L^2}(X,\mathcal E)$, is closable and its closure, whose domain is contained in ${H^1}(X,\mathcal E)$, is still given by (\ref{form}). It is in this sense that the Friedrichs extension of $\Delta$, denoted $\Delta_{W,S}$, satisfies the given mixed boundary conditions. 

In fact, under Assumptions \ref{assumpwb} and \ref{assump} it is not hard to check that $\Delta_{W,S}|_{\mathcal D_S(\mathcal E)}$ is essentially selfadjoint so we may appeal to the spectral theorem to canonically construct the associated heat semigroup
\[
e^{-\frac{1}{2}t\Delta_{W,S}}:{L^2}(X,\mathcal E)\to {L^2}(X,\mathcal E),\quad t>0. 
\] 
In particular, if $\phi\in \mathcal D_S(\mathcal E)$ then $\phi_t=e^{-\frac{1}{2}t\Delta_{R,S}}\phi$ solves the corresponding heat equation:
\begin{equation}\label{heateq}
\frac{\partial\phi_t}{\partial t}+\frac{1}{2}\Delta_{W,S}\phi_t=0, \quad \lim_{t\to 0}\phi_t=\phi, \quad \Pi_+(\nabla_\nu-S)\phi_t=0,\quad \Pi_-\phi_t=0.
\end{equation}
Of course, it is precisely this semigroup that appears in Definition \ref{consforms}.

\begin{remark}\label{bdconditions}
	{\rm 
The most general kind of (differential) boundary conditions for generalized Laplacians takes the form
\begin{equation}\label{tangcond}
A\phi=0,\qquad C\phi+\sum_{j=1}^{n-1}E_j\nabla_{e_j}\phi+E\nabla_\nu\phi=0,
\end{equation}	
where  $\{e_j\}$ is a local orthonormal frame along $\Sigma$ and the coefficents (capital letters) are locally defined matrices acting on the components of $\phi$ and $\nabla_{e_j}\phi$. In this setting, the so-called Lopatinskij-Shapiro ellipticity condition reduces to verifying that the $\mathbb C$-linear map
\[
\mathfrak b_{(\xi,z)}\phi=
\left(
\begin{array}{c}
A\phi\\
\left(i\sum_jE_j\xi_j-\sqrt{|\xi|^2-z}\,E\right)\phi
\end{array}
\right)
\] 
is an isomorphism onto its image for any $(0,0)\neq (\xi,z)\in T^*\Sigma\times\mathcal K$, where $\mathcal K=\mathbb C-(0,+\infty)$ and we use here an appropriate branch for the square root; see \cite[Lemma 1.4.8]{Gi}. As expected, the matrix $C$ plays no role here, since only the symbol of the differential term in the second condition in (\ref{tangcond}) really matters. 
The usage of stochastic methods in Section \ref{semigroup}, which relies on certain curvature driven multiplicative functionals, forces us to choose the coefficients so as to eliminate the tangential derivatives in (\ref{tangcond}) while still keeping
ellipticity; compare to Remark \ref{finalr}. Given these constraints, we are basically led to set  $A=\Pi_-$, $C=-\Pi_+S$, $E=\Pi_+$ and $E_j=0$ as in Definition \ref{mixedcond}, so that  
\[
\mathfrak b_{(\xi,z)}\phi=
\left(
\begin{array}{c}
\Pi_-\phi\\
-\sqrt{|\xi|^2-z}\,\Pi_+\phi
\end{array}
\right),
\] 
is an isomorphism indeed. Notice that this latter assertion only depends on the existence of the involution $\mathcal I$, which determines the complementary projections $\Pi_\pm$. In particular, we see that the selfadjoint endomorphism $S$ only plays a role in assuring that the quadratic form $Q$ in  Proposition \ref{propmixed} is symmetric.  
}
\end{remark}

\section{The semigroup domination property}\label{semigroup}

In this section we prove the main technical result in the paper, namely, the domination property for the heat semigroup $e^{-\frac{1}{2}t\Delta_{W,S}}$ introduced in the previous section. The crucial point here is to make sure that $e^{-\frac{1}{2}t\Delta_{W,S}}\phi\in L^1(X,\mathcal E)$ whenever $\phi\in\mathcal D_S(\mathcal E)$, with an exponential bound on the norm of the corresponding linear map depending on the lower bounds imposed on $W$ and $S$; see Corollary 4.2. A key ingredient in the proof is a Feynman-Kac formula generalizing a previous result in \cite{dL1} for differential forms, from which a path integral representation for the associated heat kernel follows. We remark that nowhere in this section Assumption \ref{assdirac} is used, so all the results here actually hold for any generalized Laplacian $\Delta_{W,S}$ satisfying Assumption \ref{assumpwb}.  

Let ${\sf X}_t={\sf X}_t^{x}$, $t\geq 0$, be reflected Brownian motion on $X$ starting at some ${x}$. 
Since Assumption \ref{assump} is taken for granted, by Theorem \ref{consesp} we know that $X$ is stochastically complete (with respect to ${\sf X}_t$). In view of Proposition \ref{equiv}, this means that ${\sf X}_t$ is non-explosive, so the sample  paths ${\sf X}_t^x$ remain in $X$ for all time. 

Although this is not strictly required in the following, for simplicity we assume that $\mathcal E$ is tensorial in the sense that it is associated to some orthogonal representation $\rho$ of  ${\rm SO}_n$, the 
rotation group in dimension $n$. As a consequence, any section $\phi\in\Gamma(X,\mathcal E)$ can be identified to its $\rho$-equivariant lift $\phi^\dagger:P_{\rm SO}(X)\to V$, where $V$ is the representation space of $\rho$. Also, the heat operator $L$ in (\ref{heateq}) lifts to 
\[
L^\dagger=\frac{\partial}{\partial t}+\frac{1}{2}\Delta_{W,S}^\dagger, 
\]   
where 
\[
\Delta_{W,S}^\dagger=\nabla^*\nabla^\dagger+W^\dagger,
\]
and 
\[
\nabla^*\nabla^\dagger=-\sum_{i=1}^n\mathcal L_{H_i}^2
\]
is the {\em horizontal} Bochner Laplacian. Here, $\mathcal L$ is Lie derivative. Also, the boundary conditions in (\ref{mixedcond2}) lift to 
\[
\Pi_+^\dagger(\mathcal L_{\nu^\dagger}-S^\dagger)\phi^\dagger=0,\quad \Pi_-^\dagger\phi^\dagger=0. 
\] 
The advantage of lifting everything in sight to $P_{\rm SO}(X)$ is that, when doing computations in the framework of It\^o's stochastic calculus, we may work on the trivial vector bundle $\mathbb R^N\to\mathbb R^n$, $N={\rm rank}\,\mathcal E$,  where the anti-development  of $\widetilde {\sf X}_t$ lives (as already mentioned, this happens to be the standard Brownian motion $b_t$ in $\mathbb R^n$); see \cite{IW,El,Hs1} for details on this so-called Eells-Elworthy-Malliavin approach to diffusions on manifolds.

We use this formalism to obtain a stochastic representation for the action of the heat semigroup $e^{-\frac{1}{2}t\Delta_{W,S}}$ on $\mathcal D_S(\mathcal E)$; see Theorem \ref{mainfk} below. 
We start by observing that for 
each $W\in{L^2_{\rm loc}}(M,{\rm End}(\mathcal E))$ and $S\in{L^2_{\rm loc}}(\Sigma,{\rm End}(\mathcal E|_{\Sigma}))$ we may consider 
the pathwise solution $M_{W,S,t}\in{\rm End}(\mathbb R^N)$ of 
\begin{equation}\label{eqmf1}
dM_{W,S,t}+M_{W,S,t}\left(\frac{1}{2}W^\dagger({\widetilde {\sf X}}_t)dt+S^\dagger({\widetilde {\sf X}}_t) d\lambda_t\right)=0, \quad M_{W,S,0}=I;
\end{equation}
see \cite{DF}. 
Note that  the inverse process $M_{W,S,t}^{-1}$  satisfies
\begin{equation}\label{eqmf12}
dM_{W,S,t}^{-1}-\left(\frac{1}{2}W^\dagger({\widetilde {\sf X}}_t)dt+S^\dagger({\widetilde {\sf X}}_t) d\lambda_t\right)M_{W,S,t}^{-1}=0, \quad M_{W,S,0}^{-1}=I.
\end{equation}

For each $\epsilon>0$ and $S$ as above defining mixed boundary conditions let us set 
\[
S^\epsilon=S+\epsilon^{-1}\Pi_-.
\]
Notice that 
\begin{equation}\label{elimin}
(S^\epsilon)^\dagger\phi^\dagger=S^\dagger\phi^\dagger, \quad \phi\in\mathcal D(\mathcal E). 
\end{equation}
Also, in the following $\|\,\|$ is the operator norm in ${\rm End}(\mathbb R^N)$. 

\begin{proposition}\label{estimgeo}
	If Assumption \ref{assumpwb} holds and if $\epsilon>0$ satisfies $\epsilon^{-1}\geq c_2$ then
	\[
	\|M_{W,S^\epsilon,t}\|\leq\exp\left(-\frac{1}{2}\int_0^tw(\widetilde{\sf X}_s)ds-\int_0^t\sigma(\widetilde{\sf X}_s)d\lambda_s\right),
	\]
	where $w$ and $\sigma$ are given by (\ref{defr}) and (\ref{defsigma}), respectively.
	\end{proposition}

\begin{proof}
	The key point here is to make sure that the righthand side does { not} depend on $\epsilon$, so it can be further estimated solely in terms of the lower bounds on $W$ and $S$. 
	Following \cite{Hs2}, we observe that it suffices to prove the result for $M_{W,S^\epsilon,t}^{\bullet}$, where the bulllet means transposition. Take $v\in\mathbb R^N$ and set $f(t)=|M_{W,S^\epsilon,t}^{\bullet}v|^2$. Then
	\begin{eqnarray*}
		df(t) & = & -2v^\bullet M_{W,S_\epsilon,t}\left(\frac{1}{2}W^\dagger({\widetilde {\sf X}}_t)dt+(S^\epsilon)^\dagger({\widetilde {\sf X}}_t) d\lambda_t\right)M_{W,S_\epsilon,t}^{\bullet}v\\
		&\leq & -f(t)\left(w(\widetilde{\sf X}_t)dt+2\sigma(\widetilde{\sf X}_t)d\lambda_t\right),
		\end{eqnarray*}
	and the result follows after integration.
\end{proof}

The following proposition is a key technical ingredient in our argument. It allows us to establish a Feynman-Kac formula for $e^{-\frac{1}{2}t\Delta_{W,S}}$ under the more restrictive assumption that $W$ and $S$ are {\em uniformly bounded}, i.e. bounded from above and below; see Theorem \ref{feyn-kac-form} below. 

\begin{proposition}\label{fkprep}
	Take $W$ and $S$ as above, with both being {uniformly bounded} and with $S$ defining mixed boundary conditions. Then, as $\epsilon\to 0$,  $M_{W,S^\epsilon,t}$ converges in $L^2$ to an adapted, right-continuous process $M_t$ with left limits. Furthermore, \begin{equation}\label{heateqlift}
	M_t\Pi^\dagger_-(\widetilde{\sf X}_t)=0,
	\end{equation} 
	whenever $\widetilde {\sf X}_t\in\pi^{-1}\Sigma$.  
	\end{proposition}

\begin{proof}
	This has been first proved in \cite{Hs2} for $1$-forms, i.e.  $\mathcal E=\wedge^1T^*X$, $W={\rm Ric}$ and  $S=B$, under the assumption  that $X$ is compact. It has been observed in \cite{dL1} that the same proof works  for $p$-forms on a non-compact manifold with bounded geometry in the sense of \cite{Schi}; hence, using the notation in Subsection \ref{diffforms}, in this case we take $\mathcal E=\wedge^pT^*X$, $W=R_p$ and $S=\mathcal B_p$. As a careful analysis of the original proof confirms, the same argument still works fine if more generally $X$ has controlled geometry in  the sense of Assumption \ref{assump}, so that 
	the integrability result in Theorem \ref{intloct} holds, and both $W$ and $S$ are uniformly bounded. We leave the details to the interested reader. 
\end{proof} 

Now, let $\phi\in\mathcal D_S(\mathcal E))$, so that 
$\phi_t^\dagger=e^{-\frac{1}{2}t\Delta^\dagger_{W,S}}\phi^\dagger$
is the solution to 
\begin{equation}\label{solhom}
L^\dagger\phi^\dagger_t=0,\quad \lim_{t\to +\infty}\phi^\dagger_t=\phi^\dagger,\quad  \Pi_+^\dagger(\mathcal L_{\nu^\dagger}-S^\dagger)\phi_t^\dagger=0,\quad \Pi_-^\dagger\phi_t^\dagger=0.
\end{equation} 
Then a simple application of It\^o's formula to the process $M_{W,S^\epsilon,t}\phi^\dagger_{T-t}({\widetilde {\sf X}}_t)$, $0\leq t\leq T$,  yields in the limit $\epsilon\to 0$ the following fundamental Feynman-Kac formula, which generalizes \cite[Theorem 5.2]{dL1}. 

\begin{theorem}
	\label{feyn-kac-form}
				Assume that $W$ and $S$ are as above, with both being {uniformly} bounded and with $S$ defining mixed boundary conditions. 
		Then  
	\begin{equation}
	\label{feyn-kac2}
	\phi^\dagger_t(\widetilde x)=\mathbb E_{{\widetilde x}}(M_t\phi^\dagger({\widetilde {\sf X}}_t^{x})).
	\end{equation}
	Equivalently, 
	\begin{equation}
	\label{feyn-kac3}
	(e^{-\frac{1}{2}t\Delta_{W,S}}\phi)(x)=\mathbb E_{x}(M_tJ_t\phi({\sf X}_t^{x})),
	\end{equation}
	where $J_t$ is the (reversed) stochastic parallel transport acting on sections of $\mathcal E$ and we use the standard identification $\phi^\dagger=J_t\phi$.    
\end{theorem}

\begin{proof}
	With the help of (\ref{stocut}),	It\^o's formula gives 
	\begin{eqnarray*}
		d M_{W,S^\epsilon,t}\phi^\dagger_{T-t}(\widetilde {\sf X}_t^{x}) & = & 
		\left\langle M_{W, S^\epsilon,t}\mathcal L_{H}\phi^\dagger_{T-t}(\widetilde {\sf X}_t^{x}),db_t\right\rangle
		- M_{W,S^\epsilon,t} L^\dagger\phi^\dagger_{T-t}(\widetilde {\sf X}_t^{x})dt\\
		& & \quad + M_{W,S^\epsilon,t}\left(\mathcal L_{\nu^\dagger}-S^\dagger-\epsilon^{-1}\Pi_{-}^\dagger\right)\phi^\dagger_{T-t}(\widetilde {\sf X}_t^{x})d\lambda_t,
	\end{eqnarray*} 
where $\mathcal L_H=(\mathcal L_{H_1},\cdots,\mathcal L_{H_n})$.
Due to (\ref{solhom}), both the second term and the term involving $\epsilon^{-1}$ on the righthand side vanish. 
Sending $\epsilon\to 0$ and using Proposition \ref{fkprep} we end up with 
	\begin{eqnarray*}
	d M_{t}\phi^\dagger_{T-t}(\widetilde {\sf X}_t^{x}) & = & 
	\left\langle M_{t}\mathcal L_{H}\phi^\dagger_{T-t}(\widetilde {\sf X}_t^{x}),db_t\right\rangle\\
	& & \quad + M_{t}\Pi_{+}^\dagger\left(\mathcal L_{\nu^\dagger}-S^\dagger\right)\phi^\dagger_{T-t}(\widetilde {\sf X}_t^{x})d\lambda_t,
	\end{eqnarray*} 	
	where the insertion of  $\Pi^\dagger_{+}$ in the last term is justified by (\ref{heateqlift}).  
	Again by (\ref{solhom}) this reduces to   	
	\[
	dM_t\phi^\dagger_{T-t}(\widetilde {\sf X}_t^{x})  =  \left\langle M_t\mathcal L_{H}\phi^\dagger_{T-t}(\widetilde {\sf X}_t^{x}),db_t\right\rangle,
	\]
	thus showing that $M_t\phi^\dagger_{T-t}(\widetilde {\sf X}_t^{x})$ is a (local) martingale. The result now follows by equating   expectations of this process at $t=0$ and $t=T$.
\end{proof}

Our aim now is to extend the Feynman-Kac formula  (\ref{feyn-kac3}) 
to the case in which $R$ and $S$ are merely assumed to be bounded from below; see Theorem \ref{mainfk} below. For this we rely on the results above to implement an approximation scheme adapted from \cite{Gu2}; see also \cite{DF} for similar arguments. 

We start with a comparison estimate holding in the general context of solutions of (\ref{eqmf1}). In order to simplify the notation in the following we sometimes write $w(t)=w({\widetilde{\sf X}_t})$, $W_1(t)=W_1({\widetilde{\sf X}_t})$, etc. Also, recall that $N={\rm rank}\,\mathcal E$.

\begin{proposition}\label{gron}
	For each $t>0$ we have the pathwise estimate
	\begin{eqnarray*}
	\|M_{W_1,S_1,t}-M_{W_2,S_2,t}\|& \leq &  e^{\int_0^t\left(\frac{1}{2}\|W_1(s)\|ds+\|S_1(s)\|d\lambda_s\right)+
	2\int_0^t\left(\frac{1}{2}\|W_2(s)\|ds+\|S_2(s)\|d\lambda_s\right)}\times\\
	& & \qquad\times \int_0^t\left(\frac{1}{2}\|W_1(s)-W_2(s)\|ds+\|S_1(s)-S_2(s)\|d\lambda_s\right).
	\end{eqnarray*}
	\end{proposition}

\begin{proof}
	From (\ref{eqmf1}) and (\ref{eqmf12}),
	\begin{eqnarray*}
	{d}\left(M_{W_2,S_2,t}^{-1}M_{W_1,S_1,t}\right)& = &M_{W_2,S_2,t}^{-1}\times\\
	& & \times  \left(\frac{1}{2}\left(W_2(t)-W_1(t)\right)dt+\left(S_2(t)-S_1(t)\right){d\lambda_t}\right)M_{W_1,S_1,t},
	\end{eqnarray*}
	so that 
	\[
	\begin{array}{l}
	M_{W_1,S_1,t}    =  M_{W_2,S_2,t}+\\
	\quad +M_{W_2,S_2,t}\int_0^tM_{W_2,S_2,s}^{-1}
	\left(\frac{1}{2}\left(W_1(s)-W_2(s)\right)dt+\left(S_1(s)-S_2(s)\right){d\lambda_s}\right)M_{W_1,S_1,s}.
	\end{array}
	\]
Thus, 
\[
\begin{array}{l}
\|	M_{W_1,S_1,t} - M_{W_2,S_2,t}\|\leq \|M_{W_2,S_2,t}\|\times\\
\quad\times \int_0^t\|M_{W_2,S_2,s}^{-1}\|M_{W_1,S_1,s}\|
\left(\frac{1}{2}\left\|W_1(s)-W_2(s)\right\|ds+\left\|S_1(s)-S_2(s)\right\|{d\lambda_s}\right).
\end{array}
\]
The result now follows since we can easily estimate the norms $ \|M_{W_i,S_i,t}\|$ and $ \|M_{W_i,S_i,t}^{-1}\|$ in the indicated way via Gronwall's inequality.
\end{proof}

Now we will be able to implement the approximation scheme. So we consider $W$ and $S$, both bounded from below. 
Define a sequence $\{W_i\}$ by setting $W_i={\min}\{W,i\,{\rm Id}\}$ fiberwise and similarly for $\{S_i\}$. It follows that $W_i$ and $S_i$ are uniformly bounded and  $\|W_i(x)- W(x)\|\to 0$ and $\|S_i(x)- S(x)\|\to 0$ as $i\to +\infty$, $x\in X$. Also, the convergences are monotone nondecreasing in the obvious sense. Moreover, as a result of this procedure we see that any $\phi\in\mathcal D_S(\mathcal E)$ can be written as $\phi=\lim_{i\to +\infty}\phi_i$, $\phi_i\in\mathcal D_{S_i}(\mathcal E)$. 

\begin{proposition}\label{estmult}
	For each $t>0$ and $\epsilon>0$ we have the pathwise convergence
	\[
	\lim_{i\to+\infty}\|M_{W_i,S_i^\epsilon,t}-M_{W,S^\epsilon,t}\|=0
	\]
\end{proposition}

\begin{proof}
	From Proposition \ref{gron} and the nondecreasing monotone convergence, 
\begin{eqnarray*}
	\|M_{W_i,S_i^\epsilon,t}-M_{W,S^\epsilon,t}\|& \leq &  e^{3\int_0^t\left(\frac{1}{2}\|W(s)\|ds+\|S^\epsilon(s)\|d\lambda_s\right)}\times\\
	& & \qquad\times \int_0^t\left(\frac{1}{2}\|W_i(s)-W(s)\|ds+\|S_i(s)-S(s)\|d\lambda_s\right).
	\end{eqnarray*}
Consider ${\sf w}=\|W\|\in L^2_{\rm loc}(X)$ and ${\sf s}^\epsilon=\|S^\epsilon\|\in L^2_{\rm loc}(\Sigma)$. It is well-known that for any $t>0$ and almost every path ${\sf X}_s^x$ we have 
\[
\int_0^t|{\sf w}({\sf X}_s^x)|ds<+\infty\quad {\rm and}\quad \int_0^t|{\sf s}^\epsilon({\sf X}_s^x)|d\lambda_s<+\infty.
\]
Thus, 
\[
\|M_{W_i,S_i^\epsilon,t}-M_{W,S^\epsilon,t}\|\leq 
C_{t,\epsilon}\int_0^t\left(\frac{1}{2}\|W_i(s)-W(s)\|ds+\|S_i(s)-S(s)\|d\lambda_s\right),
\]
and the result follows by dominated convergence.
\end{proof}

\begin{proposition}\label{estmultcor} 
	For each $\epsilon>0$, 
	\[
	\lim_{i\to +\infty}\mathbb E_x\|M_{W_i,S_i^\epsilon,t}-M_{W,S^\epsilon,t}\|^2=0.
	\]
	\end{proposition}

\begin{proof}
	Let $\{f_1,\cdots,f_N\}$ be an orthonormal frame locally trivializing $\mathcal E$ and set $Z_{i,t}^\epsilon=M_{W_i,S_i^\epsilon,t}-M_{W,S^\epsilon,t}$. 
We have 
	\begin{eqnarray*}
	d\|Z_{i,t}^\epsilon f_\alpha\|^2 & = & -2\left\langle
	Z_{i,t}^\epsilon\left(\frac{1}{2}(W(t)-W_i(t))dt+(S(t)-S_i(t))d\lambda_t\right)f_\alpha
	,Z_{i,t}^\epsilon f_\alpha\right\rangle\\
	& \leq & -2\left(\frac{1}{2}w^{(i)}(t)dt+\sigma^{(i)}(t)d\lambda_t\right)\|Z_{i,t}^\epsilon f_\alpha\|^2,
	\end{eqnarray*}
where $W-W_i\geq w^{(i)}{\rm Id}$ and $S-S_i\geq \sigma^{(i)}{\rm Id}$.
Recalling that the convergences $W_i\to W$ and $S_i\to S$ are monotone nondecreasing,
we may assume that both $w^{(i)}$ and $\sigma^{(i)}$ are nonnegative, so  $d\|Z_{i,t}^\epsilon f_\alpha\|^2\leq 0$ and hence $\|Z_{i,t}^\epsilon\|^2\leq 1$. The result then follows from Proposition \ref{estmult} and dominated convergence.
\end{proof}

We know from Proposition \ref{fkprep} that for each $i$,  $M_{W_i,S_i^\epsilon,t}$ converges in $L^2$ to a process, say $M_{i,t}$, as $\epsilon\to 0$. Moreover, by Theorem \ref{feyn-kac-form} this leads to a  Feynman-Kac formula, namely, 
\begin{equation}\label{fkpasslim}
(e^{-\frac{1}{2}t\Delta_{W_i,S_i}}\phi)(x)=\mathbb E_x\left[M_{i,t}J_t\phi({\sf X}_t^{x})\right], \quad \phi\in\mathcal D_{S_i}(\mathcal E).
\end{equation}
Now set $\mathbb E_x^{(2)}\|M_{i,t}-M_{j,t}\|=(\mathbb E_x\|M_{i,t}-M_{j,t}\|^2)^{1/2}$, etc. Then  Proposition \ref{estmultcor} and the triangle inequality
\begin{eqnarray*}
	\mathbb E_x^{(2)}\|M_{i,t}-M_{j,t}\| & \leq & \mathbb  E_x^{(2)}\|M_{i,t}-M_{W_i,S_i^\epsilon,t}\| \\
	& & \quad + \mathbb E_x^{(2)}\|M_{W_i,S_i^\epsilon,t}-M_{W_j,S_j^\epsilon,t}\|+\mathbb E_x^{(2)}\|M_{W_j,S_j^\epsilon,t}-M_{j,t}\|
	\end{eqnarray*}
imply that $\{M_{i,t}\}_i$ is Cauchy in $L^2$, so it converges as $i\to +\infty$ to a process, say $\mathcal M_t$. 
Passing the limit in (\ref{fkpasslim}) and  making use of a standard result on the monotone convergence of quadratic forms \cite[Theorem 3.18]{LHB} 
we
obtain a Feynman-Kac formula for the heat semigroup $e^{-\frac{1}{2}t\Delta_{W,S}}$. 

\begin{theorem}\label{mainfk}
	If Assumption \ref{assumpwb} holds then  
	\begin{equation}\label{limfkform}
	(e^{-\frac{1}{2}t\Delta_{W,S}}\phi)(x)=	 	\mathbb E_x\left[\mathcal M_tJ_t\phi({\sf X}_t^x)\right],\quad \phi\in\mathcal D_S(\mathcal E).
	\end{equation}
\end{theorem}

This immediately yields a path integral representation for the heat kernel  $K_{W,S}$ of $e^{-\frac{1}{2}t\Delta_{W,S}}$.

\begin{theorem}\label{desintfin}
	We have 
	\begin{equation}\label{desintfin2}
	K_{W,S}({t;x,y})=K_0(t;x,y)\mathbb E_{t;x,y}\left[\mathcal M_tJ_t\right],
	\end{equation}
	where here $J_t$ is the stochastic parallel transport along the (reversed) reflected Brownian bridge path joining $y$ to $x$. 
	\end{theorem}

\begin{proof}
	If $\phi_i\in \mathcal D_{S_i}(\mathcal E)$ then 
	\[
	M_{i,t}J_t\phi_i({\sf X}_t^x)=\int_XK_{W_i,S_i}(0;{\sf X}_t^x,y)M_{i,t}J_t\phi_i(y)dX_y.
	\]
	By taking expectation and using  (\ref{lawrel}) and Proposition \ref{extprop},
	\begin{eqnarray*}
		\mathbb	E_{x}\left[M_{i,t}J_t\phi_i({\sf X}_t^x)\right] 
		& = & \int_X K_0(t;x,y)\left(\mathbb E_x\left[
		\frac{K_0^{\otimes^{N^2}}(0;{\sf X}_t^x,y)}{K_0(t;x,y)}M_{i,t}J_t\phi_i(y)\right]\right)dX_y\\
		& = & 
		\int_XK_0(t;x,y)\mathbb E_{t;x,y}\left[M_{i,t}J_t\phi_i({\mathfrak X}_{t;x,y})\right]dX_y,
	\end{eqnarray*}
	so after passing the limit we get   
	\[
	\mathbb	E_{x}\left[\mathcal M_tJ_t\phi({\sf X}_t^x)\right] =
	\int_XK_0(t;x,y)\mathbb E_{t;x,y}\left[\mathcal M_tJ_t
	\phi(\mathfrak X_{t;x,y})\right]dX_y, \quad \phi\in\mathcal D_S(\mathcal E).
		\]
	On the other hand, from (\ref{limfkform}) we have 
	\[
		\mathbb	E_{x}\left[\mathcal M_tJ_t\phi({\sf X}_t^x)\right] 	 =\int_XK_{W,S}(t;x,y)\phi(y)dX_y.
	\]
Since $\mathfrak X_{t;x,y}=y$ and  $\phi$ is arbitrary, the result follows.
\end{proof}

Finally, we can establish the semigroup domination property for $K_{W,S}$.

\begin{theorem}\label{domsem0}
	If Assumption \ref{assumpwb} holds then  there exist $C_1, C_2>0$  such that 
	\begin{equation}\label{domsem00}
		\left\|K_{W,S}(t;x,y)\right\|\leq C_1e^{C_2t}K_0(t;x,y), 
	\end{equation}
	for any $t>0$, $x,y\in X$ and $\phi\in \mathcal D_S(\mathcal E)$.
\end{theorem}

\begin{proof}
	It suffices to prove that 
	\[
	\begin{array}{l}
	\left|\int_X\int_X\left\langle K_{W,S}(t;x,y),\phi(x)\otimes\psi(y)\right\rangle dX_xdX_y\right|\leq C_1e^{C_2t}\times\\
	\\
	\quad \times\int_X K_0(t;x,y)|\phi(x)||\psi(y)|dX_xdX_y, 
	\end{array}
	\]
	where $\phi,\psi\in\mathcal D(\mathcal E)$, and then send  $\phi\otimes \psi\to\delta_x^{\otimes^N}\otimes\delta_y^{\otimes^N}$. For this first note that 
	\[
	\begin{array}{l}
	\left|\int_X\int_X\left\langle K_{W,S}(t;x,y),\phi(x)\otimes\psi(y)\right\rangle dX_xdX_y\right|\\
\\ 
\quad	
	\leq \int_X\left|\int_X\left\langle K_{W,S}(t;x,y)\phi(x)dX_x,\psi(y)\right\rangle\right|dX_y\\
	\\
	\quad =\int_X\left|\int_XK_0(t;x,y)\left\langle\mathbb E_{t;x,y}[\mathcal M_tJ_t\phi(x)]dX_x,\phi(y)\right\rangle\right|dX_y,
	\end{array}
	\]
	where we used (\ref{desintfin2}) in the last step.
	On the other hand,
	since $J_t$ is an isometry, Proposition \ref{estimgeo} implies
	\[
	\left|\mathbb E_{t;x,y}[M_{W_i,S_i^\epsilon,t}J_t]\right|\leq C_Ne^{-\frac{1}{2}c_{1,i}t}|\mathbb E_{t;x,y}[e^{-c_{2,i}\lambda_t}]|,
	\]
	where $c_{1,i}{\rm Id}$ and $c_{2,i}{\rm Id}$ are lower bounds for $W_i$ and $S_i$, respectively. By sending $\epsilon\to 0$ we get 
	\[
	\left|\mathbb E_{t;x,y}[M_{i,t}J_t]\right|\leq C_Ne^{-\frac{1}{2}c_{1,i}t}|\mathbb E_{t;x,y}[e^{-c_{2,i}\lambda_t}]|.
	\]
	Clearly, we may assume that $c_{1,i}\to c_1$ and $c_{2,i}\to c_2$ as $i\to +\infty$ and that $c_2<0$, so after taking the limit in $i$ we may apply Theorem \ref{intloct} and the ensuing discussion with $c_2\geq p\underline\kappa$ for some $p\in[1,+\infty)$ to get  
	\[
	\left|\mathbb E_{t;x,y}[\mathcal M_{t}J_t\phi(x)]\right|\leq C_1e^{C_2t}|\phi(x)|, 
	\]
	for  $C_1=C_NK_1^{(p)}$ and $C_2=-c_1/2+K_2^{(p)}$.
	This clearly proves the integral inequality above and completes the proof.
	\end{proof}

\begin{corollary}\label{vanishres}
	If we may take $c_2=0$ then 
	\[
	\left\|K_{W,S}(t;x,y)\right\|\leq C_1e^{-\frac{1}{2}c_1t}K_0(t;x,y).
	\]
	In particular, if $c_1>-\lambda_0$, where $\lambda_0$ is the botton of the spectrum of $\Delta_0$, then $\mathcal H(\mathcal E)\cap {L^2}(X,\mathcal E)=\{0\}$. 
	\end{corollary}

\begin{proof}
If $c_2=0$ then $\mathbb E_{t;x,y}[e^{-c_2\lambda_t}]\leq 1$ and it is clear from the proof above that we may take $C_2=-c_1/2$, so the estimate on $\|K_{W,S}\|$ follows. From this the vanishing result can be easily obtained by means of a well-known argument \cite{ER, Ro2}.	
\end{proof}

\begin{corollary}\label{domsem}
	There exist $C_1, C_2>0$  such that 
	\begin{equation}\label{domsem2}
	\left\|e^{-\frac{1}{2}t\Delta_{W,S}}\phi\right\|_{{L^1}(X,\mathcal E)}\leq C_1e^{C_2t}\left\|\phi\right\|_{{L^1}(X,\mathcal E)},  
	\end{equation}
	for any $t>0$ and $\phi\in \mathcal D_S(\mathcal E)$.
	\end{corollary}

\begin{proof}
From (\ref{domsem00}) we have 
	\begin{eqnarray*}
		|(e^{-\frac{1}{2}t\Delta_{W,S}}\phi)(x)| & = & \left|\int_XK_{W,S}(t;x,y)\phi(y)dX_y\right| \\
		& \leq & C_1e^{C_2t}\int_XK_{0}(t;x,y)|\phi(y)|dX_y\\
		& = & C_1e^{C_2t}(e^{-\frac{1}{2}t\Delta_0}|\phi|)(x),  
		\end{eqnarray*}
	and after integration we obtain
	\[
	\left\|e^{-\frac{1}{2}t\Delta_{W,S}}\phi\right\|_{{L^1}(X,\mathcal E)}\leq C_1e^{C_2t}\left\|e^{-\frac{1}{2}t\Delta_0}|\phi|\right\|_{L^1(X)}\leq C_1e^{C_2t}\left\|\phi\right\|_{{L^1}(X,\mathcal E)},
	\]
	where in the last step we used that $e^{-\frac{1}{2}t\Delta_0}$ defines a contraction on $L^1(X)$. 
\end{proof}

As we shall see below, this semigroup domination property is going to play a key role in the proof of our main result. 

\section
{The proof of Theorem \ref{main}}\label{proofmain}

In this section we present the proof of Theorem \ref{main} following the lines of the argument in  \cite{M}. We start with an useful integral identity.

\begin{proposition}
	\label{intident} Let $\phi\in \mathcal D_S(\mathcal E)$ and $\xi\in{\rm Dom}(\Delta_{W,S})$.  Then, for any $t>0$,
	\begin{equation}\label{intident2}
	\left(e^{-\frac{1}{2}t\Delta_{W,S}}\phi-\phi,\xi\right)=-\frac{1}{2}\int_0^t\int_X\langle e^{-\frac{1}{2}\tau\Delta_{W,S}}\phi,\Delta_{W,S}\xi\rangle dXd\tau.
	\end{equation}
	\end{proposition}

\begin{proof}
	We compute:
	\begin{eqnarray*}
		\left(e^{-\frac{1}{2}t\Delta_{W,S}}\phi-\phi,\xi\right)
		& = & \int_X\langle e^{-\frac{1}{2}t\Delta_{W,S}}\phi
		-e^{-\frac{1}{2}0\Delta_{W,S}}\phi,\xi\rangle dX\\
		& = & \int_0^t\int_X\langle \partial_\tau e^{-\frac{1}{2}\tau\Delta_{W,S}}\phi,\xi\rangle dXd\tau\\
		& \stackrel{(\ref{heateq})}{=} & 
		-\frac{1}{2}\int_0^t\int_X\langle \Delta_{W,S}e^{-\frac{1}{2}\tau\Delta_{W,S}}\phi,\xi\rangle dXd\tau\\
		& {=} &
		-\frac{1}{2}\int_0^t\int_X \langle e^{-\frac{1}{2}\tau\Delta_{W,S}}\phi, \Delta_{W,S}\xi\rangle dXd\tau,
		\end{eqnarray*}
where we used Proposition \ref{propmixed} in the last step.
\end{proof}

We now take a sequence of smooth, compactly supported functions $h_i$ on $X$ such that $0\leq h_i\leq h_{i+1}\leq 1$, $h_i\to {\bf 1}$ as $i\to +\infty$ and $\partial h_i/\partial\nu=0$ along $\Sigma$. 

\begin{proposition}\label{auxx}
If Assumption \ref{assump} is satisfied then  
\[
\zeta_i(x)=\int_0^{+\infty}e^{-t}\int_XK_0(t;x,y)h_i(y)dX_ydt, \quad x\in X,
\]
is smooth and satisfies: a) $\zeta_i\to {\rm 1}$; b) $\frac{1}{2}\Delta_0\zeta_i=h_i-\zeta_i\to 0$; and c) $\partial \zeta_i/\partial \nu=0$ along $\Sigma$. 
\end{proposition}

\begin{proof}
In fact we only use that $X$ is stochastically complete by Theorem \ref{consesp}. By Proposition \ref{equiv}, (2),  we have 
\[
\zeta_i(x)-1=\int_0^{+\infty}e^{-t}\int_XK_0(t;x,y)\left(h_i(y)-1\right)dX_ydt,
\]	
from which $a)$ follows easily. Also,
\begin{eqnarray*}
	\frac{1}{2}\Delta_0\zeta_i(x) & = & \int_0^{+\infty}e^{-t}\int_X\frac{1}{2}\Delta_0K_0(t;x,y)h_i(y)dX_ydt\\
	& = & -\int_X\left(\int_0^{+\infty}e^{-t}\frac{\partial}{\partial t}K_0(t;x,y)dt\right)h_i(y)dX_y\\
	& = & -\int_X\left(-K_0(0;x,y)+\int_0^{+\infty}e^{-t}K_0(t;x,y)dt\right)h_i(y)dX_y,
	\end{eqnarray*}
which yields $b)$. The proof of $c)$ is obvious.
\end{proof}

We now have all the ingredients needed in the proof of Theorem \ref{main}.
Indeed, take $\phi$ as in Definition \ref{consforms}
and 
$\xi=\zeta_i\eta$, where $\eta$ is as in Definition \ref{consforms}. Since $\partial\zeta_i/\partial\nu=0$, $\zeta_i\eta\in {\rm Dom}(\Delta_{W,S})$. Also, since $\eta$ is harmonic, we may use Assumption \ref{assdirac} 
together with  (\ref{fundrel}) to check that $D(\zeta_i\eta)=D_{\sf c}\zeta_i\cdot\eta$, so that 
\[
\Delta_{W,S} (\zeta_i\eta)=D_{\sf c}^2\zeta_i\cdot\eta
=(\Delta_0\zeta_i)\eta.
\]  
Hence, from Proposition \ref{intident} and Corollary \ref{domsem} we get for each $t>0$,
\begin{eqnarray*}
\left|\left(e^{-\frac{1}{2}t\Delta_{W,S}}\phi-\phi,\zeta_i\eta\right)\right|
& \leq & \frac{1}{2}\|\Delta_0 \zeta_i\|_{L^\infty(X)}\|\eta\|_{{L^\infty}(X,\mathcal E)}\int_0^t\|e^{-\frac{1}{2}\tau\Delta_{W,S}}\phi\|_{{L^1}(X,\mathcal E)}d\tau\nonumber\\
& \leq &  \frac{C_1}{2}\|\Delta_0 \zeta_i\|_{L^\infty(X)}\|\eta\|_{{L^\infty}(X,\mathcal E)}\|\phi\|_{L^1(X,\mathcal E)}
\int_0^te^{C_2\tau}d\tau.
\end{eqnarray*}
By sending $i\to +\infty$, Proposition \ref{auxx} guarantees that the righthand side goes to $0$. Since $\zeta_i\eta\to\eta$ we obtain (\ref{consforms2}), which completes the proof of Theorem \ref{main}.

\section{Some examples}\label{examples}

In this 
section we indicate a few applications of our  results to some generalized Laplacians appearing in Geometry. As always, we assume that Assumption \ref{assump} is satisfied by the base manifold $(X,g)$.

\subsection{The Hodge Laplacian}\label{diffforms}
For $0\leq p\leq n$ we denote by $\mathcal A^p(X)=\Gamma(X,\wedge^pT^*X)$ the space of differential $p$-forms on $X$. Let $d$ be the exterior differential acting on forms and $d^\star=\pm\star d\star$ be the  co-differential, where $\star$ is the Hodge star operator. 

Recall that the Hodge Laplacian acting on $p$-forms is given by 
\begin{equation}\label{sqdirac}
\Delta_p=(d+d^\star)^2=dd^\star+d^\star d. 
\end{equation}
This is a generalized Laplacian due to the so-called Weitzenb\"ock decomposition, namely, 
\[
\Delta_p=\nabla^*\nabla_{p}+R_p,
\] 
where $\nabla^*\nabla_{p}$ is the Bochner Laplacian associated to the standard Levi-Civita connetion on $\wedge^pT^*M$ and $R_p$ is the {Weitzenb\"ock operator}, a (pointwise) selfadjoint element in $\Gamma(X,{\rm End}(\wedge^pT^*X))$  whose local expression depends on the  curvature tensor of $(X,g)$ \cite{Ro2}. We note that $R_1={\rm Ric}$.
Also, 
recall that the Clifford bundle ${\sf Cl}(TX)$ may be viewed as a Dirac bundle over itself under left Clifford multiplication. Moreover, under 
the standard vector bundle identification $\wedge T^*X={\sf Cl}(TX)$, one has $D_{\sf c}=d+d^*$ \cite[Chapter II, Theorem 5.12]{LM}, so $\Delta_p$ is a generalized Dirac Laplacian by (\ref{sqdirac}).

To implement boundary conditions in this setting we note that, 
given $\alpha\in\mathcal A^p(X)$, its restriction to $\Sigma$ decomposes into its tangential and normal components, namely, 
\begin{equation}\label{decortabs}
\alpha=\alpha_{\rm t}+\alpha_{\rm n}.
\end{equation}

\begin{definition}\label{absrel}
	We say that a $p$-form $\alpha$ is {\rm absolute } if $\alpha_{\rm n}=0$ and $(d\alpha)_{\rm n}=0$. 
\end{definition}

In turns out that the differential condition in Definition \ref{absrel} can be expressed in terms of the shape operator $B=-\nabla_\nu$  of $\Sigma$.  
To see this, extend $B$  to $TM|_\Sigma$ by declaring that $B\nu=0$ and then extend this further to $\wedge^pT^*X|_{\Sigma}$ as the selfadjont operator $\mathcal B_p$ given by
\[
(\mathcal B_p\alpha)(e_1,\cdots,e_p)=\sum_i\alpha(e_1,\cdots,Be_i,\cdots, e_p),
\]
where $\{e_i\}$ is a local orthonormal frame. 
Notice that $\mathcal B_p$ preserves the decomposition given by (\ref{decortabs}). More precisely, if $\Pi_{\rm t}$ and $\Pi_{\rm n}$  
denote the orthogonal projections onto the tangential and normal factors, respectively, with the corresponding orthonormal bundle decomposition $\wedge^pTX|_\Sigma=\mathcal F_{\rm t}\oplus\mathcal F_{\rm n}$, then $\mathcal B_p$ commutes with both projections. In particular, if $\alpha$ is absolute then $\mathcal B_p\alpha\in\Gamma(\Sigma,{\mathcal F_{\rm t}})$. 

If we choose $e_i$ so that $Be_j=\kappa_je_j$, $j=1,\cdots,n-1$, where $\kappa_j$ are the principal curvatures of $\Sigma$, it is immediate to check that 
\[
(\mathcal B_p\alpha)(e_{j_1},\cdots,e_{j_p})=\left(\sum_k\kappa_{j_k}\right)\alpha(e_{j_1},\cdots,e_{j_p}), \quad \alpha\in\Gamma(\mathcal F_{\rm t}),
\]
which shows that the sums in the brackets are the  eigenvalues of $\mathcal B_p|_{\mathcal F_{\rm t}}$. The remarks above allow us to redefine $\mathcal B_p$ so that $\mathcal B_p|_{\mathcal F_{\rm n}}=0$.

The next result shows that absolute boundary conditions are of mixed type.

\begin{proposition}\cite[Proposition 5.1]{dL1}
	\label{abscond2}
	A differential $p$-form $\alpha$ is absolute if and only if 
	\begin{equation}
	\label{absbound3}
	\Pi_{\rm t}(\nabla_\nu-\mathcal B_p)\alpha=0, \quad 	\Pi_{\rm n}\alpha=0.
	\end{equation}
\end{proposition}

This discussion shows that if we take $\mathcal F_+=\mathcal F_{\rm t}$, $\mathcal F_-=\mathcal F_{\rm n}$ and $S=\mathcal B_p$, and of course if we assume that both $R_p$ and $\mathcal B_p$ are bounded from below then the general setting in Sections \ref{feyn-kac}
and \ref{semigroup} applies here. In particular, we have the corresponding heat semigroup  
$e^{-\frac{1}{2}t\Delta_{R_p,\mathcal B_p}}$ at our disposal. 

To apply Theorem \ref{main} in this setting, it remains to check that $\Delta_{R_p,\mathcal B_p}$ satisfies Assumption \ref{assdirac}. This   is related to the remarkable fact that the quadratic form associated to the Hodge Laplacian $\Delta_p$ on $\mathcal D(\wedge^pT^*X)$ is always nonnegative, irrespective of the existence of lower bounds for $R_p$ and $\mathcal B_p$. This is already suggested by (\ref{sqdirac}), which expresses the Hodge Laplacian as the square of the Dirac operator $D=d+d^\star$. The formal proof uses the integrated version of (\ref{sqdirac}), namely, 
\begin{eqnarray*}
	\int_X\langle \Delta_p\alpha,\alpha\rangle\,dX
	& = & 
	\int_X\left(|d\alpha|^2+|d^\star\alpha|^2\right)dX\nonumber\\
	& & \qquad +\int_\Sigma \left((d^\star\alpha)_{\rm t}\wedge\star\alpha_{\rm n}-\alpha_{\rm t}\wedge\star(d\alpha)_{\rm n}\right),
\end{eqnarray*}	
so if  $\alpha$ is absolute we end up with  
\begin{equation}\label{preself2}
\int_X\langle \Delta_p\alpha,\alpha\rangle\,dX= 
\int_X\left(|d\alpha|^2+|d^\star\alpha|^2\right)dX,
\end{equation}	
which shows that $Q$ is {nonnegative}. Moreover, if $\Delta_p\alpha=0$ then $d\alpha=0$ and $d^*\alpha=0$ so that $D\alpha=0$, as desired.

\begin{remark}\label{nonneg}{\em We should emphasize that even though $Q$ is nonnegative, uniform lower bounds on $R_p$ and $\mathcal B_p$ are still required in order to obtain the semigroup domination property corresponding to Theorem \ref{domsem0} in this setting. A counterexample may be found by adapting the elementary construction  in \cite{St}. This yields a manifold $X$ for which 
		\[
		e^{-\frac{1}{2}t\Delta_1}(\mathcal D_B(\wedge^1T^*X))\subsetneq L^1(X,\wedge^1T^*X),
		\]
		 which clearly contradicts Collorary \ref{domsem}.
	}
\end{remark}

To rephrase Theorem \ref{main} in this setting we attach to the curvature invariants above the functions
\[
r_{(p)}:X\to \mathbb R, \quad r_{(p)}(x)=\inf_{|\alpha|=1}\langle R_p(x)\alpha,\alpha\rangle, 
\]
and 
\[
\kappa_{(p)}:\Sigma\to\mathbb R, \quad \kappa_{(p)}(x)=\inf_{1\leq j_1<\cdots<j_p\leq n-1}\kappa_{j_1}(x)+\cdots+\kappa_{j_p}(x).
\]
With this notation at hand we can state the  main result of this subsection, which is a straightforward application of Theorem \ref{main}. 

\begin{theorem}\label{mainform}
	If Assumption \ref{assump} is satisfied and  for some $1\leq p\leq n-1$ we have 
	$r_{(p)}\geq c_1$ for some $c_1>-\infty$
	then the heat conservation principle holds for $\Delta_{R_p,\mathcal B_p}$.
\end{theorem}

\begin{proof}
	Use that $\kappa_{(p)}\geq c_2>-\infty$ because $B$ is bounded from below in view of Assumption \ref{assump}. 
\end{proof}

\begin{corollary}\label{cormain}
	If Assuption \ref{assump} is satisfied then the heat conservation principle holds for $\Delta_{R_1,B}$. 
\end{corollary}

\begin{proof}
	Combine Theorem \ref{mainform} with Theorem \ref{consesp} and observe that  here both lower bounds $r_{(1)}\geq c_1$ and $\sigma_{(1)}\geq c_2$ already follow from Assumption \ref{assump}. 
\end{proof}

\begin{remark}\label{vanishforms}
	{\rm From Corollary \ref{vanishres} we obtain a vanishing result for absolute $L^2$ harmonic $p$-forms under the assumptions that $c_1>-\lambda_0$ and $\Sigma$
is (weakly) $p$-convex in the sense that 
\[
\inf_{x\in \Sigma}\kappa_{(p)}(x)\geq 0.
\] 
This strengthens \cite[Theorem 5.3]{dL1}, where the result was obtained under the assumption that $X$ has bounded geometry.}
	\end{remark}

\begin{remark}\label{function}
	{\rm A simpler variant of the argument leading to Theorem \ref{main}, which dispenses with Proposition \ref{auxx}, yields a proof of Theorem \ref{consesp}. We first note that by geodesic completeness  we may assume that $\|dh_i\|_{{L^\infty}(X,\wedge^1T^*X)}\to 0$. Thus, using (\ref{intident2}) with $\phi=f$ a function as in item (4) of Proposition \ref{equiv} and $\xi=h_i$ we have
		\begin{eqnarray*}
			\left(e^{-\frac{1}{2}t\Delta_0}f-f,h_i\right)_0 & = & -\frac{1}{2}\int_0^t\int_X\langle e^{-\frac{1}{2}\tau\Delta_0}f,\Delta_0 h_i\rangle dXd\tau\\
			& = & -\frac{1}{2}\int_0^t\int_X\langle e^{-\frac{1}{2}\tau\Delta_0}f,d^*dh_i\rangle dXd\tau\\
			& = & -\frac{1}{2}\int_0^t\int_X\langle de^{-\frac{1}{2}\tau\Delta_0}f,dh_i\rangle dXd\tau\\
			& = & -\frac{1}{2}\int_0^t\int_X\langle e^{-\frac{1}{2}\tau\Delta_{R_1,B}}df,dh_i\rangle dXd\tau,
		\end{eqnarray*}	
		where here we assume that $t<{\bf e}$, the extinction time of ${\sf X}_t$. It follows from Theorem \ref{domsem} applied to $1$-forms that 
		\begin{eqnarray*}
		\left|	\left(e^{-\frac{1}{2}t\Delta_{R_1,B}}f-f,h_i\right)_0\right| & \leq & 
		\frac{1}{2}\|dh_i\|_{{L^\infty}(X,\wedge^1T^*X)}\int_0^t\|e^{-\frac{1}{2}\tau\Delta_{R_1,B}}df\|_{{L^1}(X,\wedge^1T^*X)}d\tau\nonumber\\
		& \leq & \frac{C_1}{2}\|dh_i\|_{{L^\infty}(X,\wedge^1T^*X)}\|df\|_{L^1(X,\mathcal E)}\int_0^te^{C_2\tau}d\tau.
		\end{eqnarray*}
		By sending $i\to +\infty$ we then recover item (4) in Proposition  \ref{equiv} for some $t>0$, which proves Theorem \ref{consesp}. Note that in order to avoid circularity in the argument, it is crucial here not using the functions $\zeta_i$ in Proposition \ref{auxx}. Finally, we observe that the argument above is a  concrete manifestation of an abstract reasoning in \cite[Theorem 3.2.6]{BGL}. 
	} 
\end{remark}

\subsection{The Dirac Laplacian}
\label{spin}

Let $X$ be a ${\rm spin}^c$ manifold and fix a ${\rm spin}^c$ structure. In \cite[Section 5]{dL1} it is proved a Feynman-Kac formula for the semigroup $e^{-\frac{1}{2}t\Delta}$ associated to the Dirac Laplacian $\Delta=D^2$, where here $D$ is the Dirac operator acting on spinors associated to a metric $g$ on $X$ and a unitary connection on the auxiliary complex line bundle $\mathcal U$. This formula was established under the assumption that the pair $(X,\Sigma)$ has bounded geometry  and by imposing suitable boundary conditions on spinors along $\Sigma$. As a consequence, a semigroup domination result for $e^{-\frac{1}{2}t\Delta}$ was derived in this setting. We now show that   more generally, i.e. under Assumption \ref{assump}, we may also  derive  a semigroup domination inequality for $e^{-\frac{1}{2}t\Delta}$ under suitable mixed boundary conditions. As a consequence we will show that the corresponding heat conservation principle for $\Delta$ holds. 

Let $\mathbb SX=P_{{\rm Spin}^c}(X)\times_{\zeta} V$ be the ${\rm spin}$ bundle of $X$, where $\zeta$ is the complex spin representation \cite{Fr,LM}. Thus, $P_{{\rm Spin}^c}(X)$ is a ${\rm Spin}_n^c$-principal bundle  double covering $P_{{\rm SO}}(X)\times P_{{\rm U}_1}(X)$, where $P_{{\rm U}_1}(X)$ is the ${\rm U}_1$-principal bundle associated to $\mathcal U\to X$,  so  the Levi-Civita connection  on $TX$ induces a  metric  connection on $\mathbb SX$, still denoted $\nabla$. The corresponding Dirac operator $D:\Gamma(X,\mathbb SX)\to\Gamma(X,\mathbb SX)$ is locally given by
$$
D\psi=\sum_{i=1}^n \gamma(e_i)\nabla_{e_i}\psi,\quad \psi\in\Gamma(X,\mathbb SX),
$$
where $\{e_i\}_{i=1}^n$ is a local orthonormal frame and $\gamma:TX\to{\rm End}(\mathbb SX)$ is the Clifford product by tangent vectors.
In this setting, the  Dirac Laplacian operator $\Delta=D^2$ satisfies the Lichnerowicz decomposition
\begin{equation}\label{lich}
\Delta=\nabla^*\nabla+\mathfrak R, \quad \mathfrak R=\frac{\varrho}{4}+ \frac{1}{2}\gamma(i\Theta),
\end{equation}
where 
$\varrho$ is the scalar curvature of $X$ and $i\Theta$ if the curvature $2$-form of the given unitary connection on $\mathcal F$. Clearly, this is a generalized Dirac Laplacian.

In the presence  of the boundary we must also consider the restricted spin bundle   $\mathbb SX|_{\Sigma}$. By defining the restricted Clifford product and the restricted connection by 
$$
\gamma^{\intercal}(X)\psi=\gamma(X)\gamma(\nu) \psi,\quad X\in \Gamma(\Sigma,T\Sigma), \quad \psi\in \Gamma(\Sigma,\mathbb SX|_{\Sigma}),
$$ 
and 
\begin{equation}\label{conn0}
\nabla^{\intercal}_X\psi  =  \nabla_X\psi-\frac{1}{2}\gamma^{\intercal}(BX)\psi,
\end{equation}
respectively, where as usual $B=-\nabla\nu$ is the shape operator of  $\Sigma$, then  $\mathbb SX|_\Sigma$ becomes a Dirac bundle over ${\sf Cl}(TX|_{\Sigma})$ \cite{HMZ,NR}.
The associated Dirac operator $D^{\intercal}:\Gamma(\Sigma,\mathbb SX|_{\Sigma})\to\Gamma(\Sigma,\mathbb SX|_{\Sigma})$ is
$$
D^{\intercal}=\sum_{j=1}^{n-1}\gamma^{\intercal}(e_j)\nabla^{\intercal}_{e_j},
$$
where the frame has been adapted so that $e_n=\nu$.

To see the relevance of this tangential Dirac operator, assume $Be_j=\kappa_je_j$, where $\kappa_j$ are the principal curvatures of $\Sigma$. It follows that 
$$
D^{\intercal}=\frac{H}{2}+
\sum_{j=1}^{n-1}\gamma(e_j)\nabla_{e_j},
$$
where $H={\rm tr}\,B$ is the mean curvature. Hence,  ${\sf D}=-\gamma(\nu)D$ is given by
\begin{equation}\label{dirt1}
{\sf D}=D^{\intercal}+\nabla_\nu-\frac{H}{2}. 
\end{equation}

We now specify mixed boundary conditions in this setting. We start with an involutive endomorphism $\mathcal I\in\Gamma(X|_\Sigma,\mathbb SX)$, which we extend to a collared neighborhood of $\Sigma$ such that $\nabla_\nu\mathcal I=0$. Let $\Pi_{\pm}$ be the corresponding projections and set $\mathcal F_\pm=\Pi_\pm\mathbb SX|_\Sigma$. In particular, $\nabla_\nu\Pi_{\pm}=\Pi_{\pm}\nabla_\nu$. We now recall a notion introduced in \cite{dL1}. 

\begin{definition}\label{compcond} 
	We say that the tangential Dirac operator $D^\intercal$ intertwines the projections if 
	$\Pi_\pm D^\intercal=D^\intercal\Pi_{\mp}$.
\end{definition}

If this compatibility condition between $D^\intercal$ and $\Pi_\pm$ holds and $\psi,\eta\in \Gamma(\Sigma,\mathcal F_+)$ then $\langle D^\intercal \psi,\eta\rangle=0$ and hence, by (\ref{dirt1}), 
\begin{eqnarray}\label{bdspin}
	\langle {\sf D}\psi,\eta\rangle & = & \left\langle\left(\nabla_\nu-\frac{H}{2}\right)\psi,\eta\right\rangle\nonumber\\
	& = & \left\langle\Pi_+\left(\nabla_\nu-\frac{H}{2}\right)\psi,\eta\right\rangle
	\end{eqnarray}
Thus, we may proceed as in the proof of Proposition \ref{propmixed} to get
\begin{eqnarray*}
	\int_\Sigma\langle\nabla_\nu\psi,\eta\rangle d\Sigma
		& = & \int_\Sigma\left\langle{\sf D}\psi,\eta\right\rangle d\Sigma+\int_\Sigma\frac{H}{2}\langle\psi,\eta\rangle d\Sigma\\
	& = & \int_\Sigma\left\langle\Pi_+\left(\nabla_\nu-\frac{H}{2}\right)\psi,\eta\right\rangle d\Sigma+\int_\Sigma\frac{H}{2}\langle\psi,\eta\rangle d\Sigma. 
	\end{eqnarray*}
If we think of $H$ as an endomorphism $\widehat H$ of $\mathbb SX|_\Sigma$ such that $\widehat H=H\,{\rm Id}$ on $\mathcal F_+$ and $\widehat H=0$ on $\mathcal F_-$, and impose the mixed boundary conditions 
\begin{equation}
\label{bdconddir}
 \Pi_+\left(\nabla_\nu-\frac{\widehat H}{2}\right)\psi=0,\qquad \Pi_-\psi=0, 
\end{equation}
then for compactly supported spinors $\psi$ and $\eta$ satisfying these conditions we see that the bilinear form associated to $\Delta$ satisfies
\[
Q(\psi,\eta)=\int_X\langle\nabla\psi,\nabla\eta\rangle dX+\int_X\langle\mathfrak R\psi,\eta\rangle dX + \frac{1}{2}\int_\Sigma\langle {\widehat H}\psi,\eta\rangle d\Sigma. 
\]
Clearly, this is symmetric and bounded from below if $\mathfrak R$ and $H$ are uniformly bounded from below. It follows from (\ref{lich}) and Assumption \ref{assump} that $\mathfrak R$ is bounded from below if and only if so does $i\Theta$. Moreover, $H$ is always bounded from below.

To apply Theorem \ref{main} in this setting, it remains to check that Assumption \ref{assdirac} is satisfied. To see this, take $\psi\in\Gamma(X,\mathbb SX)$ compactly supported and recall that the corresponding Green's formula holds, namely, 
\begin{equation}\label{greenf}
\int_X\langle \Delta\psi,\psi\rangle dX=\int_X|D\psi|^2dX+\int_\Sigma\langle {\sf D}\psi,\psi\rangle d\Sigma. 
\end{equation}
Thus, if $\psi$ satisfies (\ref{bdconddir}) then $\langle {\sf D}\psi,\psi\rangle=0$ by (\ref{bdspin}), so we get 
\[
\int_X\langle \Delta\psi,\psi\rangle dX=\int_X|D\psi|^2dX, 
\]
and hence $\Delta\psi=0$ implies $D\psi=0$, as desired. 
Thus, as a consequence of Theorem \ref{main} we obtain the following result.

\begin{theorem}\label{mainspin}
	Let  $X$ be a ${\rm spin}^c$ manifold satisfying Asumption \ref{assump} and assume that $i\Theta$ is bounded from below. Then the heat conservation principle holds for $\Delta$.
	\end{theorem}

We note that examples of boundary conditions satisfying (\ref{bdconddir}) include both chilarity and  MIT bag boundary conditions; see Remarks 5.1 and 5.2 in  \cite{dL1}.

\begin{remark}\label{finalr}
	{\rm It is worthwhile to observe that the Green's formula in (\ref{greenf}) holds for {\em any} generalized Dirac Laplacian as long as we define ${\sf D}=-\nu\cdot D$. In particular, we see that Assumption \ref{assdirac} holds whenever we impose the boundary condition
	\[
	{\sf D}\psi=0, \quad \psi\in\Gamma(X,\mathcal E|_{\Sigma}). 
	\]
However, in general this is {\em not} a mixed boundary condition according to Definition \ref{mixedcond}. In fact, the whole point of the intertwining condition in Definition \ref{compcond} is to make sure that this is the case for the Dirac Laplacian acting on spinors. We refer to \cite[Remark 5.3]{dL1} for a discussion of this issue in the context of the Hodge Laplacian considered in the previous subsection.}
	\end{remark}

\begin{remark}\label{vanishspinors}
	{\rm From Corollary \ref{vanishres} we obtain a vanishing result for $L^2$ harmonic spinors satisfying the given boundary conditions  if we further assume that $\mathfrak R\geq c>-\lambda_0$ and $\Sigma$
		is mean convex ($H\geq 0$).
		This strengthens \cite[Theorem 5.5]{dL1}, where the result was obtained under the assumption that $X$ has bounded geometry.}
\end{remark}

\subsection{The Jacobi operator on free boundary minimal immersions}\label{freeb}
Let $(\overline X,\overline g)$ be a non-compact Riemannian manifold of dimension $\overline n>n$ and with boundary $\overline \Sigma$. 
Let $\Psi:(X,g)\looparrowright (\overline X,\overline g)$ be a non-compact isometric immersion with boundary $\Sigma=X\cap\overline \Sigma$. 
If $TX^{\perp}$ is the normal bundle of $X$, $\mathfrak B\in\Gamma(X,{{\rm Hom}(TX\otimes TX,TX^\perp)})$ is the second fundamental form of $X$. Also, we denote by $
\overline R$ the curvature tensor of $(\overline X,\overline g)$. 

Any compactly supported vector field $U\in\Gamma(X,T\overline X|_{X})$ which is {\em admissible } in the sense that it is tangent to $\overline \Sigma$ along $\Sigma$ gives rise to a one-parameter family of isometric immersions $t\in(-\varepsilon,\varepsilon)\mapsto \Psi_t:(X, g_t)\looparrowright(\overline X,\overline g)$, $\varepsilon>0$, such that $\Psi_0=\Psi$ and 
\[
\frac{\partial \Psi_t}{\partial t}|_{t=0}=U. 
\] 
We then say that $U$ is the variational field associated to the variation $\Psi_t$. 
A direct computation gives the first variation of the area functional
\[
\left(\delta_{(X,g)}{\rm Area}\right)(U)=\frac{d}{dt}{\rm Area}(X_t,g_t)|_{t=0}
\]
along a variational field $U$. We have 
\begin{equation}\label{criteqA1}
\left(\delta_{(X,g)}{\rm Area}\right)(U)=-\int_X\langle \mathcal H,U\rangle dX-\int_\Sigma\langle U,\nu\rangle d\Sigma,
\end{equation}
where $\mathcal H={\rm trace}\, \mathfrak B$ is the mean curvature vector and $\nu$ is the inward pointing unit co-normal vector along $\Sigma$. 

\begin{definition}\label{criteq}
	We say that $X$ is a free boundary minimal immersion if it is a critical point for the functional ${\rm Area}$ under compactly supported variations. 
\end{definition}

By (\ref{criteqA1}) this means that $\mathcal H=0$  along $X$ (this is the minimality condition) and $\langle U,\nu\rangle=0$ along $\Sigma$ for any $U$. This latter condition means that $\Sigma$ meets $\overline \Sigma$ orthogonally (this is the free boundary condition). Notice that this implies that $\nu$ is normal to $\overline \Sigma$. In particular, it makes sense to consider $B_{\overline \Sigma}^\nu$, the shape operator of $\overline\Sigma$ in the direction of $\nu$. 

If $(X,g)$ is a free boundary minimal immersion, it is natural to compute the second variation of the area  along {admissible} variational fields $U$ and $V$ as above. The result is
\begin{equation}\label{segvarform}
\left(\delta^2_{(X,g)}{\rm Area}\right)(U,V)=\int_X\langle \mathcal J U,V\rangle dX-\int_\Sigma\langle \left(\nabla^\perp_\nu+B_{\overline \Sigma}^\nu\right)U,V\rangle d\Sigma.
\end{equation}
Here, $\nabla^\perp$ is the normal connection on $TX^\perp$ and the {\em Jacobi operator} is given by
\[
\mathcal J=\nabla^*\nabla^{\perp}-{\sf W}, 
\]
where $\nabla^*\nabla^\perp$ is the Bochner Laplacian associated to $\nabla^\perp$, ${\sf W}={\sf R}+{\sf B}$,  ${\sf B}=\mathfrak B\circ \mathfrak B^\bullet\in\Gamma(X,{\rm End}(TX^\perp))$ and 
${\sf R}\in\Gamma(X,{\rm End}(TX^\perp))$ is given by
\[
\langle {\sf R}U,V\rangle=\sum_{i=1}^n\langle \overline R_{U,e_i}e_i,V\rangle. 
\]
Since ${\sf W}$ is clearly selfadjoint, $\mathcal J$ is a generalized Laplacian. But notice that it is {\em not} a generalized Dirac Laplacian, so a heat conservation principle corresponding to Theorem \ref{main} does not necessarily hold here; however, see Remark \ref{kahler}. 

As a consequence of (\ref{segvarform}), any Morse-theoretic notion involving this variational problem (like index, nullity, etc.) should be addressed by imposing to variational fields the Robin-type boundary condition
\begin{equation}\label{mixjac}
\left(\nabla^\perp_\nu+B_{\overline \Sigma}^\nu\right) U=0. 
\end{equation}
In particular, {\em Jacobi fields}, i.e. solutions of $\mathcal JU=0$, should be studied under this boundary condition. We refer to \cite{Scho} for details.

\begin{remark}\label{strict}{\rm 
		Note that, strictly speaking, (\ref{mixjac}) is of mixed type. Indeed, in the language of Section \ref{feyn-kac} it is obtained by taking $\mathcal I={\rm Id}$, so that $\Pi_+={\rm Id}$ and $\Pi_-=0$, and $S=-B_{\overline\Sigma}^\nu$.}
\end{remark}

Now, by (\ref{intpart}) we can rewrite (\ref{segvarform}) as 
\[
\left(\delta^2_{(X,g)}{\rm Area}\right)(U,V)=\int_X \left(\langle\nabla^\perp U,\nabla^\perp V\rangle-\langle{\sf W}U,V\rangle\right)dX-\int_\Sigma\langle B_{\overline \Sigma}^\nu U,V\rangle d\Sigma.
\]
Hence, the bilinear form
\[
Q(U,V)=\int_X\langle\mathcal JU,V\rangle dX
\]
is given by 
\begin{eqnarray*}
	Q(U,V) & = & \int_X \left(\langle\nabla^\perp U,\nabla^\perp V\rangle-\langle{\sf W}U,V\rangle\right)dX+\int_\Sigma\langle\nabla_\nu^\perp U,V\rangle d\Sigma\\
	& = & \int_X \left(\langle\nabla^\perp U,\nabla^\perp V\rangle-\langle {\sf W}U,V\rangle\right)dX+\int_\Sigma\langle(\nabla_\nu^\perp+B_{\overline \Sigma}^\nu) U,V\rangle d\Sigma\\
	&& \qquad -\int_\Sigma\langle B_{\overline\Sigma}^\nu U,V\rangle d\Sigma.
\end{eqnarray*}
Thus, $Q$ is symmetric and bounded from below if we assume that the variational fields $U$ and $V$ satisfy (\ref{mixjac}) and  impose lower bounds of the type 
\begin{equation}\label{upper}
-{\sf W}\geq c_1{\rm Id},\quad -B_{\overline\Sigma}^\nu\geq c_2{\rm Id}.
\end{equation} 
Under these assumptions, all the results in Section \ref{semigroup} hold for $\mathcal J_{-{\sf W}, -B_{\overline\Sigma}^\nu}$. In particular, the following vanishing result, corresponding to   
Corollary \ref{vanishres}, holds true.

\begin{theorem}\label{infrig}
	Under the conditions above, assume that $c_1>-\lambda_0$ and $c_2=0$ in (\ref{upper}). Then $X$  carries no $L^2$ Jacobi field satisfying (\ref{mixjac}). 
\end{theorem}

\begin{example}\label{examp}
{\rm 	Let $\overline X$ be the exterior of an open geodesic ball in hyperbolic space $\mathbb H^{\overline n}$, so that $\overline \Sigma$ is the geodesic sphere bounding this ball. Now take any totally geodesic submanifold passing through the center of the ball and take $X$ to be the portion of this submanifold outside the ball. Then Theorem \ref{infrig} clearly applies to the free boundary minimal submanifold $X$.}
	\end{example}

\begin{remark}\label{misl}{\rm We note that the proof of the domination property in this setting is substantially simplified in the sense that we can get rid of the parameter $\epsilon>0$ appearing in Section \ref{semigroup}. In fact,  this kind of simplification will take place whenever, in the notation of Secion \ref{feyn-kac}, we take $\mathcal I={\rm Id}$ as in Remark \ref{strict}. To see this, take $\phi$ satisfying (\ref{solhom}) with $\Pi_+={\rm Id}$ and $\Pi_-=0$ and directly apply It\^o's formula to $M_{W,S,t}\phi^\dagger_{T-t}({\widetilde {\sf X}}_t)$ (no mention to $\epsilon$) as in the proof of Theorem \ref{feyn-kac-form}, where we assume that both $W$ and $S$ are uniformly bounded. We end up with 
		\begin{eqnarray*}
			d M_{W,S,t}\phi^\dagger_{T-t}(\widetilde {\sf X}_t^{x}) & = & 
			\left\langle M_{W, S,t}\mathcal L_{H}\phi^\dagger_{T-t}(\widetilde {\sf X}_t^{x}),db_t\right\rangle
			- M_{W,S,t} L^\dagger\phi^\dagger_{T-t}(\widetilde {\sf X}_t^{x})dt\\
			& & \quad + M_{W,S,t}\left(\mathcal L_{\nu^\dagger}-S^\dagger\right)\phi^\dagger_{T-t}(\widetilde {\sf X}_t^{x})d\lambda_t,
		\end{eqnarray*}
		and since 		
		the last two terms vanish, $M_{W,S,t}\phi^\dagger_{T-t}({\widetilde {\sf X}}_t)$ is found to be a martingale. 
		In this way we obtain a proof of the Feynman-Kac formula in Theorem \ref{feyn-kac-form}
		without having to appeal to the rather technical $\epsilon^{-1}$-perturbation in  Propositions \ref{estimgeo} and \ref{fkprep}. From this point on we may use the approximation scheme to remove the upper bounds on  
		$W$ and $S$
		just as we did in Section \ref{semigroup}.}
\end{remark}

\begin{remark}\label{kahler}
	{\rm 
Let $(\overline X,g)$ as above be a K\"ahler manifold and assume that the free boundary 	
minimal submanifold $X\subset \overline X$ of dimension $n/2$ is {\em Lagrangian} in the sense that $\Omega|_X=0$, where $\Omega$ is the underlying symplectic form. The map that to each normal vector $u\in TX_x^\perp$ associates the $1$-form $\alpha_u=u\righthalfcup \Omega\in T^*X$ defines a bundle isomorphism between $TX^\perp$ and $T^*X$, so that to each admissible variation vector field $U\in\Gamma(X,TX^\perp)$ there corresponds a $1$-form $\alpha_U\in\mathcal A^1(X)$. If we assume further that $\overline X$ is a Ricci flat, K\"ahler-Einstein manifold, then under this identification we have $\mathcal J=\Delta_1$,
the Hodge Laplacian acting on $1$-forms \cite[Proposition 4.1]{Oh}. In particular, by Subsection \ref{diffforms}, $\mathcal J$ is a generalized Dirac Laplacian satisfying Assumption \ref{assdirac}. Recalling that $\Delta_1=\nabla^*\nabla+{\rm Ric}$ and that Assumption \ref{assump} already implies that ${\rm Ric}$ is bounded from below, an application of Theorem \ref{main} gives the following result: {\em if $-B_{\overline \Sigma}^\nu$ is bounded from below then the heat conservation principle holds for $\mathcal J$.}}   
\end{remark}

\end{document}